\documentclass[12pt,a4paper]{article}

\usepackage{graphicx}
\usepackage[]{amsmath,amssymb,amsthm}
\usepackage{paralist}
\usepackage{epstopdf}

\DeclareSymbolFont{AMSb}{U}{msb}{m}{n}
\DeclareSymbolFontAlphabet{\Bbb}{AMSb}

\usepackage[ansinew]{inputenc} 
\usepackage{latexsym}

\usepackage{amstext,amsfonts,amsmath,amsthm,graphicx,amssymb,amscd,epsfig}
\usepackage{rotating}

\usepackage{subfig}
\newtheorem{theorem}{Theorem}[section]

\newtheorem{definition}[theorem]{Definition}
\newtheorem{lemma}[theorem]{Lemma}

\newcommand{\C}{\mathbb{C}}
\newcommand{\R}{\mathbb{R}}

\newcommand{\N}{\mathbb{N}}

\newcommand{\B}{{\cal B}}

\newcommand{\T}{\rule{0pt}{2.6ex}}       

\begin{document}

\begin{center}
{\large Cyclic dominance in a two-person Rock-Scissors-Paper game} \\
\mbox{} \\
\begin{tabular}{ccc}
L.\ Garrido-da-Silva$^{1}$ & and & S.B.S.D.\ Castro$^{1,2,*} $\\
(lilianagarridosilva@sapo.pt) & & (sdcastro@fep.up.pt) \\
OrcID: 0000-0003-4294-3931 & & OrcID: 0000-0001-9029-6893
\end{tabular}
\end{center}

\vspace{1cm}
$^*$ Corresponding author: sdcastro@fep.up.pt;
phone: +351 225 571 100; fax: +351 225 505 050. 
\medskip

$^{1}$ 
Faculdade de Economia da Universidade do Porto,
Centro de Matem\'atica da Universidade do Porto (CMUP),
Rua Dr.\ Roberto Frias,
4200-464 Porto,
Portugal. 
\medskip

$^{2}$ 
Centro de Economia e Finan\c{c}as (Cef.UP),
Rua Dr.\ Roberto Frias,
4200-464 Porto,
Portugal. 
\vspace{1cm}

\begin{abstract}
The Rock-Scissors-Paper game has been studied to account for cyclic behaviour under various game dynamics. We use a two-person parametrised version of this game. The cyclic behaviour is observed near a heteroclinic cycle, in a heteroclinic network, with two nodes such that, at each node, players alternate in winning and losing. This cycle is shown to be as stable as possible for a wide range of parameter values. The parameters are related to the players' payoff when a tie occurs. 
\end{abstract}

\vspace{1cm}



\section{Introduction}\label{sec:intro}

The Rock-Scissors-Paper game (henceforth, RSP) has been used to model behaviour and learning in the framework of evolutionary game theory in both economics and the life sciences. This game has three candidate actions: Rock~(R), Scissors~(S) and Paper~(P) such that R~beats S, S~beats P and P~beats R. For the one-person or single-population (symmetric) case it provides a good two-dimensional model for convergent and oscillating dynamics, the type of dynamics depending on some parameters of the game. 

The RSP dynamics are commonly described in terms of the continuous-time replicator equations introduced by Taylor and Jonker~(1978). Such equations capture biologically a process of natural selection, and economically a process of learning  through imitation of successful behaviours.
Players decide on their strategies by comparing the payoff over possible outcomes to the average payoff.
A payoff higher than average may be achieved by selecting the strategy that produces the highest payoff (choice of the best reply) or by copying a strategy that was observed to be successful enough (choice of a better, not necessarily best, reply). More detail on these choices is given in Section~\ref{sec:RSPgame}.

During play the state of each player is a mixed strategy over the set of actions 
from the two-dimensional simplex -- the individual state space -- whose vertices are the pure strategies, i.e. the actions themselves R, S and P. 
The replicator dynamics for the symmetric RSP game assumes that the payoffs for each action are the same for all players in the interaction. 
If the initial conditions are symmetric this reduces the analysis to one player (or, one population), which has been done by Zeeman~(1980). A noteworthy result is that the one-person RSP game does not exhibit isolated limit cycles.\footnote{Other dynamics in the RSP game can lead to stable isolated limit cycles yielding periodic oscillatory dynamics of all actions that favor long-term coexistence. See, for instance, Gaunersd\"orfer and Hofbauer~(1995) with the extension to best response dynamics, and Mobilia~(2010) and Toupo and Strogatz~(2015) in the context of populations under mutations.} Generically, two kinds of robust long-term behaviour are predicted depending on the payoffs. The player state evolves either towards the unique mixed-strategy Nash equilibrium corresponding to coexistence of all three actions, 
or towards the heteroclinic cycle on the boundary of the two-dimensional simplex. 
A \emph{heteroclinic cycle} consists of equilibria of the dynamical system and solution trajectories connecting them. When a heteroclinic cycle is stable it induces persistent cycling characterised by progressively longer residing times in each equilibrium so that the player will sequentially play $\textnormal{R} \rightarrow \textnormal{P} \rightarrow \textnormal{S} \rightarrow \textnormal{R}$ never stopping.       

We are interested in the existence of cyclic dominance 
under coupled (asymmetric) replicator equations in a two-person RSP game. Cyclic dominance occurs when over time the available actions take turns in appearing dominant, leading to cyclicity (see Szolnoki {\em et al.} 2014).
Sato~\emph{et~al}.~(2002, 2005) provide numerical results for the two-person RSP game.
A dynamic is specified for each player and the game's state space is the product of two two-dimensional simplices.  Aguiar and Castro (2010) show that the game dynamics support a \emph{heteroclinic network} made of nine pairs of pure strategies reflecting all possible sequences of play along the state space. This heteroclinic network can be seen as consisting of 
\begin{enumerate}
\item[(a)] three heteroclinic cycles, one of which involves alternate \emph{win-loss} of both players, and the other two involve \emph{loss-tie} by only one of the players;
\item[(b)] two heteroclinic cycles in which play goes through \emph{loss-tie-win} by only one of the players.
\end{enumerate}

In this paper, we address stability of cyclic behaviour in the RSP interaction for two asymmetric players through the stability of a heteroclinic cycle. The asymmetry between players results from their different valuations of the payoff for a tie.
Although heteroclinic cycles in a heteroclinic network cannot be asymptotically stable, they can exhibit a strong form of stability, known as \emph{essential asymptotic stability}, first introduced by Melbourne~(1991). Our main results describe the stability properties of the three types of heteroclinic orbits in~(a) and~(b) as the payoffs for a tie vary. We show that the win-loss cycle wherein players switch to best responses is essentially asymptotically stable when the sum of payoffs for a tie is negative.
The choice of actions along this cycle is always made by switching to best responses by both players. Intuitively, in this case, a tie is not an attractive outcome for at least one player for whom the payoff is negative and hence, as an outcome, it is avoided. The win-loss cycle models the existence of alternating dominance between two players. 
On the other hand, the loss-tie cycles can be stable in a  weaker sense for certain payoff values, and the loss-tie-win cycles are never stable. These heteroclinic cycles are therefore harder, or impossible, to observe in applications and numerical experiments.
These illustrate sequences of play where one player switches to the best response while the other switches to a better -- not best -- response.

The dynamics of the two-person RSP game can be very complex and we do not attempt to find detailed specific applications. We also do not address the question of whether other sequences of play can be followed. The heteroclinic cycles above are all such that one player always switches to the best response while the other does one of three things: (i) switches to a best response, (ii) switches to a better -- not best -- response, or (iii) switches to a better and then to a best response (in this case, the timing of play is not alternate between players). Behaviour where the latter player does not choose just one of the alternatives (i)--(iii) is not considered. We venture to conjecture that these other alternatives do not possess any kind of stability. The reasons for this conjecture are better understood after reading Section~4, where we return to this point briefly in the final paragraph.

It is known that a wide variety of choices is allowed but not all may be realised in play: even though any sequence of actions (R, S or P) is possible, it is not certain that given a sequence of outcomes there will be players that choose actions in the way prescribed by the chosen sequence.
In the language of dynamical systems, we refer to this as \emph{infinite switching}, which is shown not to exist in this game (see Olszowiec 2016; Garrido-da-Silva 2018). We do however hope that this first approach can open the door to further research in the treatment of game dynamics for asymmetric contests between two players.

This article is organised as follows: the next section contains preliminary material which may be skipped by the reader familiar with the dynamics near heteroclinic networks. Section 3 describes the two-person RSP game and its heteroclinic cycles. Section 4 provides a thorough study of the stability of all the heteroclinic cycles in the dynamics. Detailed calculations are deferred to an appendix, as well as the necessary information to describe the trajectories of points near each heteroclinic cycle in the RSP network. The last section concludes.

\section{Definitions and preliminaries}
Consider a smooth vector field $f:\mathbb{R}^{n}\rightarrow\mathbb{R}^{n}$ described by a system of differential equations
\begin{equation}
\dot{x}=f\left(x\right), \quad x \in \R^n.
\label{eq:system}
\end{equation}
An \emph{equilibrium}\footnote{Equilibria are sometimes called fixed points or steady states, and nodes in the context of heteroclinic dynamics.} $\xi \in \R^n$ of \eqref{eq:system} satisfies $f\left(\xi\right)=0$. Given two equilibria $\xi_i$ and $\xi_j$ of \eqref{eq:system} a \emph{heteroclinic connection} $\left[\xi_i\rightarrow \xi_j\right]$  is a set of solution trajectories of~\eqref{eq:system} which are backward asymptotic to $\xi_i$ and forward asymptotic to $\xi_j$.

A {\em heteroclinic cycle} is a flow-invariant set $C\subset\mathbb{R}^{n}$ consisting of an ordered collection of finitely many saddle equilibria $\left\{ \xi_{1},\ldots,\xi_{m}\right\}$ and connecting trajectories $\left[\xi_j\rightarrow \xi_{j+1}\right]$, $j=1,\ldots,m$, where $\xi_{m+1}=\xi_{1}$.
A {\em heteroclinic network} is a connected union of finitely many heteroclinic cycles. 

We say that $f$ is $\Gamma$-equivariant for some finite Lie group $\Gamma$ acting orthogonally on $\mathbb{R}^{n}$ if $f\left(\gamma\cdot x\right)=\gamma\cdot f\left(x\right)$ for all $\gamma\in\Gamma$ and $x\in\mathbb{R}^{n}$. Each $\gamma\in\Gamma$ is called a symmetry of $f$. Background on differential equations with symmetry can be found in Golubitsky {\em et al.}~(1988). The symmetry of a problem can be used to simplify its study by identifying as one different objects that are related by symmetry, see Section~\ref{sec:RSPgame}.

The $\Gamma$-orbit of $x\in\mathbb{R}^{n}$ is the set 
$
\Gamma\left(x\right)=\left\{ \gamma\cdot x,\; \gamma\in\Gamma\right\}.
$
The elements in the $\Gamma$-orbit of an equilibrium of \eqref{eq:system} are also equilibria. A group orbit of an equilibrium is called a {\em relative equilibrium}.
A heteroclinic connection between two relative equilibria is itself a heteroclinic connection between two equilibria, one belonging to the outgoing relative equilibrium and the other belonging to the incoming one.

Consider a subset $S\subset\mathbb{R}^{n}$.  
The set of all $\Gamma$-orbits of $S$ is called the {\em quotient space} and is denoted by $S/\Gamma$.
The flow of $f$ restricts to a flow on the quotient space $S/\Gamma$ whenever $S$ is flow-invariant. By identifying each $\Gamma$-orbit of $S$ with a single point in the quotient space $S/\Gamma$, we can study the dynamics on $S$ via the dynamics on $S/\Gamma$. This reduces both the dimension of the state space and the number of equilibria (also called {\em nodes}) in a heteroclinic cycle or network.

In generic systems, heteroclinic connections between saddles can be broken by arbitrarily small perturbations.
A sufficient condition for preserving their structure relies on the existence of flow-invariant subspaces. We say that a heteroclinic cycle $C$ is {\em robust} if each heteroclinic connection $\kappa_{j,j+1}$ is contained in a flow-invariant subspace $P_j$ such that $\xi_j$ is a saddle and $\xi_{j+1}$ is a sink for the flow restricted to $P_j$, for $j=1,\ldots,m$.
Robust heteroclinic cycles arise naturally in equivariant systems (e.g. Field~1996) as well as in game theory and population dynamics (e.g. Hofbauer and Sigmund~1998).
In the latter flow-invariant subspaces occur in the form of extinction hyperplanes.

It is well known that heteroclinic cycles are visible in applications and numerical simulations if they are stable. Within a heteroclinic network a heteroclinic cycle is never asymptotically stable. In fact, a heteroclinic network consists of at least two heteroclinic cycles with at least one node in common. The unstable manifold of such a node takes trajectories to both heteroclinic cycles. Hence, there are points in a neighbourhood of the common node that do not follow any given heteroclinic cycle, precluding asymptotic stability. 
The heteroclinic cycle can however exhibit strong attraction properties, which have been classified into various types of stability, as follows.
These are illustrated in Figure~\ref{fig:stability}.

\begin{figure}[!htb]
\subfloat[$\sigma_{\textnormal{loc}}(x)=+\infty$]{\includegraphics[width=0.25\textwidth]{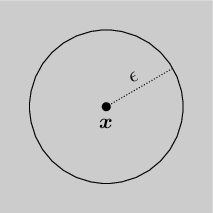}}\hfill
\subfloat[$0<\sigma_{\textnormal{loc}}(x)<+\infty$]{\includegraphics[width=0.25\textwidth]{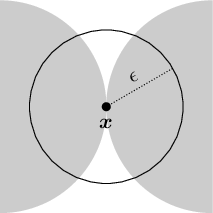}}\hfill
\subfloat[$-\infty<\sigma_{\textnormal{loc}}(x)<0$]{\includegraphics[width=0.25\textwidth]{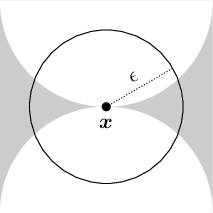}}
\caption{\label{fig:stability} The grey region represents the local basin of attraction in a neighbourhood of a point $x$ along a heteroclinic connection.
A heteroclinic cycle is asymptotically stable if for a point in each heteroclinic connection the stability index (see Definition~\ref{def:index}) $\sigma_{\textnormal{loc}}(x)=+\infty$, as in (a).
The cycle is fragmentarily asymptotically stable if there exists a heteroclinic connection along which the stability index is $-\infty < \sigma_{\textnormal{loc}}(x) < 0$ as in (c). 
The cycle is essentially asymptotically stable if one stability index is $0< \sigma_{\textnormal{loc}}(x) <+\infty$ as in~(b) and none is as in~(c).}
\end{figure}

We use terminology of Podvigina~(2012): let $S\subset\mathbb{R}^{n}$ be a compact set invariant under the flow $\Phi_{t}\left(\cdot\right)$ of the system \eqref{eq:system}, and $\epsilon,\;\delta>0$.
Given a metric $d$ on $\R^n$, we write 
$$
B_{\epsilon}\left(S\right)=\left\{ x\in \R^n: d(x,S)<\epsilon \right\}
$$
for an $\epsilon$-neighbourhood of $S$ and
$$
\B_{\delta}\left(S\right)=\left\{ x\in \mathbb{R}^n: \Phi_{t}\left(x\right)\in B_{\delta}\left(S\right) \textrm{ for any } t \geq 0
\textrm{ and } \lim_{t \rightarrow \infty}d\left(\Phi_t(x),S\right)=0 \right\}
$$
for the $\delta$-local basin of attraction of $S$.
By $\ell\left(\cdot\right)$ we denote Lebesgue measure in the appropriate dimension.

Melbourne~(1991) introduces the strongest intermediate notion of stability called essential asymptotic stability. Roughly speaking, an essentially asymptotically stable invariant object attracts all nearby trajectories except for a cuspoidal region of points sufficiently thin for making the former visible in experiments (see Figure~\ref{fig:stability}(b)). We consider the definition provided by Brannath~(1994):

\begin{definition}[Definition 1.2 in Brannath~(1994)] 
A compact invariant set $S$ is essentially asymptotically stable if there is a set $\mathcal{N}$ with $S \subset \overline{\mathcal{N}}$ ($\overline{\mathcal{N}}$ the closure of $\mathcal{N}$) such that all trajectories starting in $\mathcal{N}$ converge to $S$ without leaving a prescribed neighbourhood of $S$ and
$$
\lim_{\epsilon\rightarrow0}\frac{\ell\left(B_{\epsilon}\left(S\right)\cap \mathcal{N}\right)}{\ell\left(B_{\epsilon}\left(S\right)\right)}=1.
$$
\end{definition}

A weaker form of attractiveness is established by Podvigina~(2012) through the definition of fragmentary asymptotic stability. In this case, the set of points attracted to $S$ needs to have positive measure but can be small (see Figure~\ref{fig:stability}(c)). 
\begin{definition}[Definition 2 in Podvigina~(2012)]
A compact invariant set~$S$ is fragmentarily asymptotically stable if for any $\delta>0$
$$
\ell\left(\mathcal{B}_{\delta}\left(S\right)\right)>0.
$$
\end{definition}

Complete instability is formulated as opposed to fragmentary asymptotic stability in Podvigina~(2012). Here the probability of finding a point whose orbit remains close to the heteroclinic cycle is zero. 
\begin{definition}[Definition 3 in Podvigina (2012)]
A compact invariant set $S$ is completely unstable if there exists $\delta>0$ such that $\ell\left(\B_{\delta}\left(S\right)\right)=0$.
\end{definition}

Standard examples of completely unstable invariant sets are provided by saddle points and sources.

Podvigina and Ashwin~(2011) define an index that \emph{quantifies} the extent of the local basin of attraction of any compact invariant set $S$. 

\begin{definition}[Podvigina and Ashwin (2011)]\label{def:index}
For a compact invariant set~$S$, a point $x\in S$ and $\epsilon,\; \delta>0$, set
$$
\Sigma_{\epsilon,\delta}\left(x\right) =\frac{\ell\left(B_{\epsilon}\left(x\right)\cap\mathcal{B}_{\delta}\left(S\right)\right)}{\ell\left(B_{\epsilon}\left(x\right)\right)}.
$$
The (local) stability index of $S$ at $x$ is
$$
\sigma_{\textrm{loc}}\left(x\right)=\sigma_{\textrm{loc},+}\left(x\right)-\sigma_{\textrm{loc},-}\left(x\right),
$$
where
$$
\sigma_{\textrm{loc},-}\left(x\right) =\lim_{\delta\rightarrow0}\lim_{\epsilon\rightarrow0}\left[\frac{\ln\left(\Sigma_{\epsilon,\delta}\left(x\right)\right)}{\ln\left(\epsilon\right)}\right], \: \sigma_{\textrm{loc},+}\left(x\right)=\lim_{\delta\rightarrow0}\lim_{\epsilon\rightarrow0}\left[\frac{\ln\left(1-\Sigma_{\epsilon,\delta}\left(x\right)\right)}{\ln\left(\epsilon\right)}\right].
$$
For $\delta$ small and fixed before taking the limit, we use the convention that $\sigma_{\textrm{loc},-}\left(x\right)=+\infty$ if there is an $\epsilon_{0}>0$ such that $\Sigma_{\epsilon,\delta}\left(x\right)=0$ for all $\epsilon<\epsilon_{0}$, and $\sigma_{\textrm{loc},+}\left(x\right)=+\infty$ if $\Sigma_{\epsilon,\delta}\left(x\right)=1$ for all $\epsilon<\epsilon_{0}$. 
Therefore, $\sigma_{\textrm{loc},\pm}\left(x\right)\geq0$ and $\sigma_{\textrm{loc}}\left(x\right)\in\left[-\infty,+\infty\right]$. \end{definition}

Recall the schematic illustration in Figure~\ref{fig:stability}
where $\sigma_{\textnormal{loc}}(x)$ measures the portion of points in small enough $\epsilon$-neighbourhoods of $x \in S$ that are in the $\delta$-local basin of attraction of $S$ as the neighbourhoods shrink (grey areas in Figure~\ref{fig:stability}).

It is important to note that the stability index is constant on trajectories of the flow, see Theorem 2.2 in Podvigina and Ashwin~(2011). In particular, the stability index of a heteroclinic connection $\left[ \xi_i \rightarrow \xi_j \right]$ can be computed for an arbitrary point $x$ in $\left[ \xi_i \rightarrow \xi_j \right]$. 
We then describe attraction properties of heteroclinic cycles and networks by making use of a finite number of indices, namely the ones along their heteroclinic connections. Theorem 2.4 in Podvigina and Ashwin~(2011) allows us to reduce the dimension of the sets that need to be measured by restricting to a section transverse to the flow.

The following two results relate essential asymptotic stability and fragmentary asymptotic stability to the sign of stability indices.
Set $\ell_1\left(\cdot\right)$ for the 1-dimensional Lebesgue measure.

\begin{theorem}[Theorem 3.1 in Lohse~(2015)]\label{thm:e.a.s.}
Let $C\subset\mathbb{R}^{n}$ be a heteroclinic cycle or network with finitely many equilibria and
connecting trajectories. Suppose that $\ell_{1}\left(C\right)<\infty$ and that the stability index $\sigma_{\textrm{loc}}\left(x\right)$ exists and is not equal to zero for all $x\in C$. Then, generically, $C$ is essentially asymptotically stable if and only if $\sigma_{\textrm{loc}}\left(x\right)>0$ along all connecting trajectories.
\end{theorem}

\begin{lemma}[Lemma 2.5 in Garrido-da-Silva and Castro~(2019)]
Suppose that for $x \in C$ the stability index $\sigma_{\textrm{loc}}\left(x\right)$ is defined. If there is a point $x\in C$ such that $\sigma_{\textrm{loc}}\left(x\right)>-\infty$, then $C$ is f.a.s.
\end{lemma}

In what follows we drop the subscript $loc$ for ease of notation.

\section{The Rock-Scissors-Paper game}\label{sec:RSPgame}
We examine the long-term dynamics for the two-person Rock-Scissors-Paper (RSP) game. Each player has three possible actions R (rock), S (scissors) and P (paper) engaging in a cyclic relation: R beats S, S beats P, P beats~R.
Our description of the RSP interaction is based on Sato {\em et al.}~(2002, 2005) and Aguiar and Castro~(2010): 
two players, say $X$ and $Y$, simultaneously choose one action from $\left\{\textnormal{R}, \textnormal{S}, \textnormal{P}\right\}$. The payoff of the winning action is $+1$ while the payoff of the losing action is $-1$. 
If a tie occurs with both players choosing the same action, the respective payoffs are parametrised by quantities $\varepsilon_{x},\varepsilon_{y} \in \left(-1,1\right)$. We will assume $\varepsilon_{x}+\varepsilon_{y} \neq 0$ so that the game is not zero-sum.
The normal form representation of the game is given by two normalised payoff matrices 
$$
A=\left(\begin{array}{ccc}
0 & 1-\varepsilon_{x} & -1-\varepsilon_{x}\\
-1-\varepsilon_{x} & 0 & 1-\varepsilon_{x}\\
1-\varepsilon_{x} & -1-\varepsilon_{x} & 0
\end{array}\right), \; B=\left(\begin{array}{ccc}
0 & 1-\varepsilon_{y} & -1-\varepsilon_{y}\\
-1-\varepsilon_{y} & 0 & 1-\varepsilon_{y}\\
1-\varepsilon_{y} & -1-\varepsilon_{y} & 0
\end{array}\right),
$$
whose columns and rows respect the order of the actions: R, S, P. Each element of the matrix $A$ (resp.~$B$) is the payoff of the row player~$X$ (resp.~$Y$) playing against the column player~$Y$ (resp.~$X$). 

Within the evolution approach the players' choices are expressed in the form of state (column) vectors whose components are the probabilities of playing R, S and P.
At time $t$, these are $\textrm{\textbf{x}}=\left(x_{1},x_{2},x_{3}\right)\in \varDelta_X$ for player~$X$ and $\textrm{\textbf{y}}=\left(y_{1},y_{2},y_{3}\right)\in \varDelta_Y$ for player~$Y$, where $\varDelta_X$ and $\varDelta_Y$ denote the unit simplex in $\mathbb{R}^{3}$ associated to each player. 
The vectors $\textrm{\textbf{x}}$ and $\textrm{\textbf{y}}$ represent mixed strategies in game play.

The three vertices of  $\varDelta_X$ and $\varDelta_Y$ correspond to pure strategies for which the player assigns a probability of $1$ to each action. We then refer to those only as $\textnormal{R}=\left(1,0,0\right)$, $\textnormal{S}=\left(0,1,0\right)$ and $\textnormal{P}=\left(0,0,1\right)$.

The pair $\left(\textrm{\textbf{x}},\textrm{\textbf{y}}\right) \in \varDelta=\varDelta_{X}\times\varDelta_{Y}$ describes the state of the game at a particular time, $\varDelta$ being a four-dimensional subset of $\R^6$.
The dynamics of play evolves according to the reinforcement learning governed by the coupled replicator equations\footnote{The superscript $T$ indicates the transpose of a matrix in general.} 
\begin{equation}
\begin{aligned}\frac{dx_i}{dt} & =x_{i}\left[\left(A\textrm{\textbf{y}}\right)_{i}-\textrm{\textbf{x}}^{T}A\textrm{\textbf{y}}\right], &  & i=1,2,3,\\
\frac{dy_j}{dt} & =y_{j}\left[\left(B\textrm{\textbf{x}}\right)_{j}-\textrm{\textbf{y}}^{T}B\textrm{\textbf{x}}\right], &  & j=1,2,3,
\end{aligned}
\label{eq:replicator}
\end{equation}
where $\left(A\textrm{\textbf{y}}\right)_{i}$ and $\left(B\textrm{\textbf{y}}\right)_{j}$ are, respectively, the $i$th and $j$th element of the vectors $A\textrm{\textbf{y}}$ and $B\textrm{\textbf{x}}$. The products $\textrm{\textbf{x}}^{T}A\textrm{\textbf{y}}$ and $\textrm{\textbf{y}}^{T}B\textrm{\textbf{x}}$ represent the average payoff for players $X$ and $Y$, respectively.

The unique Nash equilibrium is
$\left(\textrm{\textbf{x}}^{*},\textrm{\textbf{y}}^{*}\right)=\left(\frac{1}{3},\frac{1}{3},\frac{1}{3},\frac{1}{3},\frac{1}{3},\frac{1}{3}\right)$ at which both players are indifferent among all three actions. This is also an equilibrium of the dynamics.
As stated in Sato \emph{et al.}~(2005), when the game is zero-sum, i.e. $\varepsilon_{x}+\varepsilon_{y}=0$, the Jacobian at the Nash equilibrium has purely imaginary eigenvalues and $\left(\textrm{\textbf{x}}^{*},\textrm{\textbf{y}}^{*}\right)$ is non-hyperbolic. 
The collective has then a constant of motion; see Equation (44) in
Sato \emph{et al.}~(2005) who provide interesting results for the zero-sum case in their Section 4.3.1.
Otherwise, if $\varepsilon_{x}+\varepsilon_{y} \neq 0$, then $\left(\textrm{\textbf{x}}^{*},\textrm{\textbf{y}}^{*}\right)$ is a saddle and trajectories may be attracted to a (robust) heteroclinic network on the boundary of~$\varDelta$. We focus on this latter case.

There are nine additional equilibria of \eqref{eq:replicator} corresponding to the vertices of~$\varDelta$, namely $\left(\textrm{\textbf{x}},\textrm{\textbf{y}}\right)$ with $\textrm{\textbf{x}},\textrm{\textbf{y}}\in\left\{\textnormal{R}, \textnormal{S}, \textnormal{P}\right\}$.  
All these equilibria are in turn saddle points.
The nine vertices together with the edges of $\varDelta$ form a heteroclinic network. See also Figure 11 in Sato {\em et al.} (2005). This heteroclinic network can be described both as the union of three 6-node heteroclinic cycles, $C_{0}$, $C_{1}$ and $C_{2}$, with
\[
\begin{aligned}C_{0} & =\left[\left(\textrm{R},\textrm{P}\right)\rightarrow\left(\textrm{S},\textrm{P}\right)\rightarrow\left(\textrm{S},\textrm{R}\right)\rightarrow\left(\textrm{P},\textrm{R}\right)\rightarrow\left(\textrm{P},\textrm{S}\right)\rightarrow\left(\textrm{R},\textrm{S}\right)\rightarrow\left(\textrm{R},\textrm{P}\right)\right]\\
C_{1} & =\left[\left(\textrm{R},\textrm{S}\right)\rightarrow\left(\textrm{R},\textrm{R}\right)\rightarrow\left(\textrm{P},\textrm{R}\right)\rightarrow\left(\textrm{P},\textrm{P}\right)\rightarrow\left(\textrm{S},\textrm{P}\right)\rightarrow\left(\textrm{S},\textrm{S}\right)\rightarrow\left(\textrm{R},\textrm{S}\right)\right]\\
C_{2} & =\left[\left(\textrm{S},\textrm{R}\right)\rightarrow\left(\textrm{R},\textrm{R}\right)\rightarrow\left(\textrm{R},\textrm{P}\right)\rightarrow\left(\textrm{P},\textrm{P}\right)\rightarrow\left(\textrm{P},\textrm{S}\right)\rightarrow\left(\textrm{S},\textrm{S}\right)\rightarrow\left(\textrm{S},\textrm{R}\right)\right],
\end{aligned}
\]
and as the union of two 9-node heteroclinic cycles, $C_{3}$ and $C_{4}$, with
\[
\begin{aligned}C_{3}= & \left[\left(\textrm{R},\textrm{S}\right)\rightarrow\left(\textrm{R},\textrm{R}\right)\rightarrow\left(\textrm{R},\textrm{P}\right)\rightarrow\left(\textrm{S},\textrm{P}\right)\rightarrow\left(\textrm{S},\textrm{S}\right)\rightarrow\left(\textrm{S},\textrm{R}\right)\rightarrow\left(\textrm{P},\textrm{R}\right)\rightarrow\right.\\
 & \left.\rightarrow\left(\textrm{P},\textrm{P}\right)\rightarrow\left(\textrm{P},\textrm{S}\right)\rightarrow\left(\textrm{R},\textrm{S}\right)\right]\\
C_{4}= & \left[\left(\textrm{S},\textrm{R}\right)\rightarrow\left(\textrm{R},\textrm{R}\right)\rightarrow\left(\textrm{P},\textrm{R}\right)\rightarrow\left(\textrm{P},\textrm{S}\right)\rightarrow\left(\textrm{S},\textrm{S}\right)\rightarrow\left(\textrm{R},\textrm{S}\right)\rightarrow\left(\textrm{R},\textrm{P}\right)\rightarrow\right.\\
 & \left.\rightarrow\left(\textrm{P},\textrm{P}\right)\rightarrow\left(\textrm{S},\textrm{P}\right)\rightarrow\left(\textrm{S},\textrm{R}\right)\right].
\end{aligned}
\]

The two alternative descriptions above allow us to identify three types of heteroclinic orbits for RSP games.
When following 
\begin{itemize}
\item the win-loss cycle, $C_0$, both players switch to best responses;
\item the loss-tie cycles, $C_1$ and $C_2$, one player switches to best response, and the other switches to better, but not best, responses mimicking the first player's strategy and resulting in a tie;
\item the loss-tie-win cycles, $C_3$ and $C_4$, one player switches to best responses, and the other switches in two steps, first to the better response and then to the best response. Players do not make their choices alternately: along $C_3$ player $X$ waits until player $Y$ has moved twice before considering a new choice of action (along $C_4$ it is player $Y$ who waits and player $X$ who makes two consecutive choices). 
\end{itemize}

The theoretic behaviour of this game was briefly investigated by Sato \emph{et al.}~(2002) via numerical simulations. Aguiar and Castro~(2010) addressed the same problem by making use of equivariant theory. 

The vector field associated to \eqref{eq:replicator} is equivariant under the action of the symmetry group $\Gamma$ generated by
$$
\gamma:\left(x_{1},x_{2},x_{3},y_{1},y_{2},y_{3}\right)\mapsto\left(x_{3},x_{1},x_{2},y_{3},y_{1},y_{2}\right).
$$
The $\Gamma$-orbits of the equilibria $\left(\textrm{R},\textrm{P}\right)$, $\left(\textrm{R},\textrm{S}\right)$ and $\left(\textrm{R},\textrm{R}\right)$ are, respectively, the following relative equilibria:
\[
\begin{aligned}\xi_{0} & \equiv\Gamma\left(\textrm{R},\textrm{P}\right)=\left\{ \left(\textrm{R},\textrm{P}\right),\left(\textrm{S},\textrm{R}\right),\left(\textrm{P},\textrm{S}\right)\right\} \\
\xi_{1} & \equiv\Gamma\left(\textrm{R},\textrm{S}\right)=\left\{ \left(\textrm{R},\textrm{S}\right),\left(\textrm{S},\textrm{P}\right),\left(\textrm{P},\textrm{R}\right)\right\} \\
\xi_{2} & \equiv\Gamma\left(\textrm{R},\textrm{R}\right)=\left\{ \left(\textrm{R},\textrm{R}\right),\left(\textrm{S},\textrm{S}\right),\left(\textrm{P},\textrm{P}\right)\right\} .
\end{aligned}
\]
These represent the three possible outcomes in the game: win, loss and tie. 
For example, player $X$ loses at $\xi_{0}$, wins at $\xi_1$ and ties at $\xi_2$. 

Due to symmetry, the dynamics of \eqref{eq:replicator} on $\R^6$ can be studied from the dynamics on the quotient space $\R^6/\Gamma$. The dynamics on $\R^6$ can be then recovered from the quotient dynamics in the usual way.
The restricted flow to the quotient  space $\R^6/\Gamma$ contains the quotient heteroclinic network with one-dimensional heteroclinic connections between two of the relative equilibria: $\xi_0$, $\xi_1$ and $\xi_2$. The quotient heteroclinic cycles are described as (see Figure \ref{fig:cycles-in-network})
\[
\begin{alignedat}{2}
\begin{aligned}C_{0} & =\left[\xi_{0}\rightarrow\xi_{1}\rightarrow\xi_{0}\right]\\
C_{1} & =\left[\xi_{1}\rightarrow\xi_{2}\rightarrow\xi_{1}\right]\\
C_{2} & =\left[\xi_{0}\rightarrow\xi_{2}\rightarrow\xi_{0}\right]
\end{aligned}
 & \quad \textrm{ and } \quad & 
 \begin{aligned}C_{3} & =\left[\xi_{0}\rightarrow\xi_{1}\rightarrow\xi_{2}\rightarrow\xi_{0}\right]\\
C_{4} & =\left[\xi_{0}\rightarrow\xi_{2}\rightarrow\xi_{1}\rightarrow\xi_{0}\right].
\end{aligned}
\end{alignedat}
\]

\begin{figure}[!htb]
\centering{}
\subfloat[$C_{0}\cup C_{1}\cup C_{2}$]{\begin{centering}\includegraphics{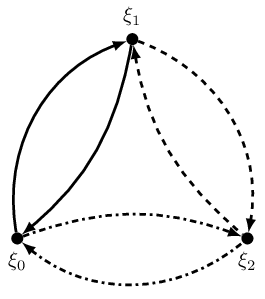}\par\end{centering}} \hspace{2cm}
\subfloat[$C_{3}\cup C_{4}$]{\begin{centering}\includegraphics{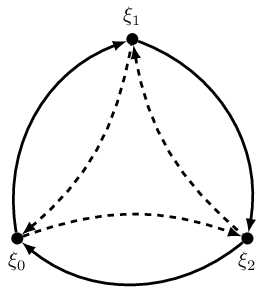}\par\end{centering}}
\caption{\label{fig:cycles-in-network}Cycles in the quotient heteroclinic network. Each style identifies a quotient heteroclinic cycle: (a)
$C_{0}$ is represented by a solid line, $C_{1}$ by a dashed line, and $C_{2}$ by a dash-dot line; (b) $C_{3}$ is represented by a solid line, and $C_{4}$ by a dashed line.}
\end{figure}

Notice that the coordinate hyperplanes as well as all sub-simplices of~$\varDelta$ are flow-invariant subspaces.
In particular, every heteroclinic connection $\left[\xi_i \rightarrow \xi_j \right]$, $i \neq j=0,1,2$, is of saddle-sink type in a two-dimensional boundary of $\varDelta$. Denote this subspace by $P_{ij}$. Evidently, $P_{ij}$ is not a vector subspace of $\R^6$.
For convenience we find a three-dimensional vector subspace of $\R^6 $, labelled $Q_{ij}$, also invariant under the flow such that $P_{ij}\subset Q_{ij}$ and $\left[\xi_i \rightarrow \xi_j \right]$ persists in a robust way. Representatives of all heteroclinic connections in the quotient heteroclinic network and the respective flow-invariant subspaces that contain them are listed in Table \ref{tab:representative}.

\begin{table}
\begin{centering}
\begin{tabular}{|c|c|c|c|}
\hline 
Connection & Representative & 2-dim space $P_{ij}$ & 3-dim vector space $Q_{ij}$\tabularnewline
\hline 
\hline 
$\left[\xi_{0}\rightarrow\xi_{1}\right]$ & $\left[\left(\textrm{R},\textrm{P}\right)\rightarrow\left(\textrm{S},\textrm{P}\right)\right]$ & $\left\{ \left(x_{1},x_{2},0;0,0,1\right)\right\} $ & $\left\{ \left(x_{1},x_{2},0;0,0,y_{3}\right)\right\} $\tabularnewline
\hline 
$\left[\xi_{1}\rightarrow\xi_{0}\right]$ & $\left[\left(\textrm{S},\textrm{P}\right)\rightarrow\left(\textrm{S},\textrm{R}\right)\right]$ & $\left\{ \left(0,1,0;y_{1},0,y_{3}\right)\right\} $ & $\left\{ \left(0,x_{2},0;y_{1},0,y_{3}\right)\right\} $\tabularnewline
\hline 
$\left[\xi_{1}\rightarrow\xi_{2}\right]$ & $\left[\left(\textrm{R},\textrm{S}\right)\rightarrow\left(\textrm{R},\textrm{R}\right)\right]$ & $\left\{ \left(1,0,0;y_{1},y_{2},0\right)\right\} $ & $\left\{ \left(x_{1},0,0;y_{1},y_{2},0\right)\right\} $\tabularnewline
\hline 
$\left[\xi_{2}\rightarrow\xi_{1}\right]$ & $\left[\left(\textrm{R},\textrm{R}\right)\rightarrow\left(\textrm{P},\textrm{R}\right)\right]$ & $\left\{ \left(x_{1},0,x_{3};1,0,0\right)\right\} $ & $\left\{ \left(x_{1},0,x_{3};y_{1},0,0\right)\right\} $\tabularnewline
\hline 
$\left[\xi_{0}\rightarrow\xi_{2}\right]$ & $\left[\left(\textrm{S},\textrm{R}\right)\rightarrow\left(\textrm{R},\textrm{R}\right)\right]$ & $\left\{ \left(x_{1},x_{2},0;1,0,0\right)\right\} $ & $\left\{ \left(x_{1},x_{2},0;y_{1},0,0\right)\right\} $\tabularnewline
\hline 
$\left[\xi_{2}\rightarrow\xi_{0}\right]$ & $\left[\left(\textrm{R},\textrm{R}\right)\rightarrow\left(\textrm{R},\textrm{P}\right)\right]$ & $\left\{ \left(1,0,0;y_{1},0,y_{3}\right)\right\} $ & $\left\{ \left(x_{1},0,0;y_{1},0,y_{3}\right)\right\} $\tabularnewline
\hline 
\end{tabular}
\par\end{centering}

\caption{\label{tab:representative}Flow-invariant subspaces and representatives for heteroclinic connections in the quotient heteroclinic network.}
\end{table}

\section{Stability of the RSP cycles}\label{sec:stability}
In this section we establish the overall stability properties of the (quotient) heteroclinic cycles $C_{p}$, $p=0,1,2,3,4$, by looking at the stability of the individual heteroclinic connections. 
The behaviour of trajectories passing close to each heteroclinic cycle is captured by \emph{Poincar\'e maps} defined on suitable cross sections to the heteroclinic connections. We construct as many Poincar\'e maps around the heteroclinic cycle as the number of its heteroclinic connections. 
They characterise accordingly the local basin of attraction of the heteroclinic cycle in a neighbourhood of each heteroclinic connection, see~\eqref{eq:basin-C0} of Appendix~\ref{sec:basin}.

Due to an appropriate change of coordinates our Poincar\'e maps can be described by a product of matrices called \emph{basic transition matrices}. Each of these matrices is related to the dynamics along one heteroclinic connection. 
All details can be found in Appendix~\ref{sec:transitions}. 

We compute the stability indices along all heteroclinic connections making up of every $C_p$-cycle by means of the approach of Garrido-da-Silva and Castro~(2019).
Their results deal with the calculation of stability indices for heteroclinic cycles comprised of one-dimensional heteroclinic connections lying in flow-invariant spaces of equal dimension. 
We observe that the $C_{p}$-cycles satisfy this assumption for any $p$.

The main tool in our analysis is the function $F^{\textrm{index}}:\R^{3} \rightarrow\left[-\infty,\infty\right]$, 
given in Definition~3.8 in Garrido-da-Silva and Castro~(2019), which can be related to the stability index in Definition~\ref{def:index} as follows: suppose that the intersection of the local basin of attraction of a compact invariant set $S$ with a cross section transverse to the flow at $x\in S$ is given by
$$
\left\{\left(x_1,x_2,x_3\right) \in \R^3: \max\left\{|x_1|,|x_2|,|x_3|\right\}<\delta \textnormal{ and }\left|x_1^{\alpha_1}x_2^{\alpha_2}x_3^{\alpha_3}\right|<1
\right\}
$$ 
for $\boldsymbol{\alpha}=\left(\alpha_1,\alpha_2,\alpha_3\right)\in \R^3$ and $\delta>0$ small. The function $\boldsymbol{\alpha}\mapsto F^{\textnormal{index}}\left(\boldsymbol{\alpha}\right)$ is the local stability index for $S$ at $x$ relative to this intersection, i.e. $F^{\textnormal{index}}\left(\boldsymbol{\alpha}\right)=\sigma_{\textnormal{loc}}(x)$.
See Appendix~\ref{sec:Findex} for the explicit form of $F^{\textrm{index}}\left(\boldsymbol{\alpha}\right)$.

For a heteroclinic cycle we denote by $\sigma_{j}$ the local stability index along the heteroclinic connection leading to the node $\xi_{j}$. For ease of reference we reproduce a result from Garrido-da-Silva and Castro (2019), which determines the stability indices for heteroclinic cycles such as those in the RSP game. 
The statement refers to Lemma \ref{lem:cond(i)-(iii)} that can be found in Appendix~\ref{sec:basin}.

\begin{theorem}[Theorem 3.10 in Garrido-da-Silva and Castro (2019)]\label{thm:stab-indices}
Let $M_{j}$, $j=1,\ldots,m$, be  basic transition matrices of a collection of maps
associated with a heteroclinic cycle. Denote by $q=j_{1},...,j_{L}$, $L\geq1$, all the indices for which $M_q$ has at least one negative entry.
\begin{enumerate}                                                                                                                                                                                                 
\item[(a)] If, for at least one $j$, the matrix $M^{\left(j\right)} = M_{j-1} \cdots M_1 M_m \cdots M_j$
does not satisfy conditions (i)-(iii) of Lemma \ref{lem:cond(i)-(iii)},
then $\sigma_{j}=-\infty$ for all $j=1,\ldots,m$ and the heteroclinic cycle is
not an attractor (it is completely unstable).
\item[(b)] If the matrices $M^{\left(j\right)}$ satisfies conditions (i)-(iii)
of Lemma \ref{lem:cond(i)-(iii)} for all $j=j_p+1$, $p=1,\ldots,L$, such that $j_p+1\notin \left\{ j_1,\ldots,j_L \right\}$,
then the heteroclinic cycle is f.a.s. Furthermore, for each $j=1,\ldots,m$, there exist vectors $\boldsymbol{\beta}_{1},\boldsymbol{\beta}_{2},\ldots,\boldsymbol{\beta}_{K}\in\R^N$,
such that 
\[
\sigma_{j}=\min_{i=1,\ldots,K}\left\{ F^{\textnormal{index}}\left(\boldsymbol{\beta}_{i}\right)\right\} .
\]
\end{enumerate}
\end{theorem}

Applying Theorem \ref{thm:stab-indices} to each $C_p$-cycle gives the stability indices of its heteroclinic connections, where the vectors $\boldsymbol{\beta}_{i}$ are the rows with negative entries of the associated basic transition matrices and their product. These negative entries occur due to the existence of two unstable directions at a node: at $\xi_0$, for instance, there is one unstable direction towards~$\xi_1$ and another towards~$\xi_2$.

Recall that the quantity $\varepsilon_{x}+\varepsilon_{y}$ is that which determines whether the game is zero-sum or not. Since we are focussing on the non-zero-sum game, this quantity is always non-zero.
\begin{theorem}\label{thm:stab-C0}
For the $C_{0}$-cycle of the RSP game and any $\varepsilon_x,\varepsilon_y\in\left(-1,1\right)$,
\begin{enumerate}
\item[(i)] if $\varepsilon_{x}+\varepsilon_{y}>0$, then all stability
indices are~$-\infty$, and the cycle is completely unstable;
\item[(ii)] if $\varepsilon_{x}+\varepsilon_{y}<0$, then the stability
indices are
\[
\begin{aligned}
\sigma_{0}=& \min\left\{\frac{1-\varepsilon_{x}}{1+\varepsilon_{x}},\dfrac{\left(1-\varepsilon_y\right)^2}{2\left(1+\varepsilon_y\right)}\right\}>0 \\
\sigma_{1}=& \min \left\{\frac{1-\varepsilon_{y}}{1+\varepsilon_{y}},\dfrac{\left(1-\varepsilon_x\right)^2}{2\left(1+\varepsilon_x\right)}\right\}>0,
\end{aligned}
\]
and the cycle is essentially asymptotically stable.
\end{enumerate}
\end{theorem}

The condition $\varepsilon_{x}+\varepsilon_{y}<0$ guarantees that the payoff for a tie is negative for at least one of the players. The payoff for winning is considerably higher than the payoff for a tie from the point of view of such a player. Hence, it makes sense that a choice of action that leads to a tie on the next round is avoided. The players only 
switch to best responses.
 
 The following result is necessary for the proof of Theorem~\ref{thm:stab-C0}.
 
 \begin{lemma}\label{lem:aux}
The transition matrix $M=M^{\left(0\right)}$ in \eqref{eq:trans-matrix-C0} of Appendix \ref{subsec:trans-mat} has one positive real eigenvalue and a pair of complex, non-real, eigenvalues.
 \end{lemma}
 
 \begin{proof}
 The eigenvalues of $M$ are the roots of the characteristic polynomial
\begin{equation}
p\left(\lambda\right)=-\lambda^{3}+\textrm{Tr}\left(M\right)\lambda^{2}-\textrm{B}\left(M\right)\lambda+\textrm{Det}\left(M\right),
\label{eq:charpoly}
\end{equation}
where $\textrm{Tr}\left(M\right)$ and $\textrm{Det}\left(M\right)$ are respectively the trace and the  determinant of $M$, and
\[
\begin{aligned}
\textrm{B}\left(M\right)= &
\left|\begin{array}{cc}
\dfrac{-1-3\varepsilon_x-\varepsilon_y+\varepsilon_x\varepsilon_y}{4} & \dfrac{1-\varepsilon_x}{2}\\[.4cm]
\dfrac{3+\varepsilon_y^2}{4} & -\dfrac{1+\varepsilon_y}{2}
\end{array}\right|\\
 & +
\left|\begin{array}{cc}
-\dfrac{1+\varepsilon_y}{2} & 0\\[.4cm]
1 & 0
\end{array}\right|+
\left|\begin{array}{cc}
\dfrac{-1-3\varepsilon_x-\varepsilon_y+\varepsilon_x\varepsilon_y}{4} & 1\\[.4cm]
\dfrac{1-\varepsilon_y}{2} & 0
\end{array}\right|.
\end{aligned}
\]
The Fundamental Theorem of Algebra states that $p(\lambda)=0$ has precisely three roots $\lambda_{1},\lambda_{2},\lambda_{3}\in \C$ such that
\begin{align}
\textrm{Tr}\left(M\right) & =\lambda_{1}+\lambda_{2}+\lambda_{3}=\frac{-3-3\varepsilon_{x}-3\varepsilon_{y}+\varepsilon_{x}\varepsilon_{y}}{4},\label{eq:tr}\\
\textrm{B}\left(M\right) & =\lambda_{1}\lambda_{2}+\lambda_{1}\lambda_{3}+\lambda_{2}\lambda_{3}=\frac{-3+3\varepsilon_{x}+3\varepsilon_{y}+\varepsilon_{x}\varepsilon_{y}}{4},\label{eq:B}\\
\textrm{Det}\left(M\right) & =\lambda_{1}\lambda_{2}\lambda_{3}=1.\nonumber
\end{align}

By virtue of the Routh-Hurwitz Criterion (see Arnold 2000), the number of roots with positive real part equals the number of sign changes of the sequence
$$
-1, \quad \textrm{Tr}\left(M\right), \quad \dfrac{1-\textrm{B}\left(M\right)\textrm{Tr}\left(M\right)}{\textrm{Tr}\left(M\right)}, \quad 1.
$$
For all $\varepsilon_x,\varepsilon_y \in \left(-1,1\right)$ we have 
$$
\textrm{Tr}\left(M\right)\textrm{B}\left(M\right)=\frac{1}{16}\left(3-\varepsilon_{x}\varepsilon_{y}\right)^{2}-\frac{9}{16}\left(\varepsilon_{x}+\varepsilon_{y}\right)^{2} \in \left(-2,1\right),
$$
yielding  
$$
sgn\left(\dfrac{1-\textrm{B}\left(M\right)\textrm{Tr}\left(M\right)}{\textrm{Tr}\left(M\right)}\right) = sgn\left(\textrm{Tr}\left(M\right)\right).
$$
Hence, there is exactly one root with positive real part (which must be a real root).

To show that the remaining eigenvalues are non-real, we look at the discriminant of $p(\lambda)$ in~\eqref{eq:charpoly}, a real cubic polynomial, which is
\[
\begin{aligned}
\Delta(& \varepsilon_{x},\varepsilon_{y})\\
& = 18\textrm{Tr}\left(M\right)\textrm{B}\left(M\right)-4\textrm{Tr}\left(M\right)^{3}+\textrm{Tr}\left(M\right)^{2}\textrm{B}\left(M\right)^{2}-4\textrm{B}\left(M\right)^{3}-27\\
& =  \begin{aligned}[t]\frac{1}{256} & \left[\left(\varepsilon_{x}^{2}-9\right)^{2}\varepsilon_{y}^{4}+\left(-80\varepsilon_{x}^{3}-432\varepsilon_{x}\right)\varepsilon_{y}^{3} \right. \\
 & +\left(-18\varepsilon_{x}^{4}-396\varepsilon_{x}^{2}-162\right)\varepsilon_{y}^{2}+\left(-432\varepsilon_{x}^{3}-3024\varepsilon_{x}\right)\varepsilon_{y}\\
 & +81\varepsilon_{x}^{4}-162\varepsilon_{x}^{2}-3375\Big].
\end{aligned}
\end{aligned}
\]
For each value of $\varepsilon_{x}\in \left(-1,1\right)$ we can regard $\Delta(\varepsilon_x,\cdot)$ as a real quartic polynomial in the variable $\varepsilon_{y}$.
Its discriminant is in turn given by
$$
-\frac{59049}{67108864}\left(\varepsilon_{x}^{2}+15\right)^{3}\left(\varepsilon_{x}^{2}+3\right)^{8}.
$$
This is negative for every $\varepsilon_{x} \in \left(-1,1\right)$ and hence $\Delta\left(\varepsilon_{x},\cdot\right)$ has two distinct real roots and two complex conjugate non-real roots. The coefficient of the leading term of $\Delta\left(\varepsilon_{x},\cdot\right)$ is positive. Together with
\[
\begin{aligned}\Delta\left(\varepsilon_{x},-1\right) & =-\frac{1}{4}\left(1-\varepsilon_{x}\right)\left(\varepsilon_{x}^{3}+9\varepsilon_{x}^{2}+54\right)<0\\
\Delta\left(\varepsilon_{x},1\right) & =-\frac{1}{4}\left(1+\varepsilon_{x}\right)\left(-\varepsilon_{x}^{3}+9\varepsilon_{x}^{2}+54\right)<0
\end{aligned}
\]
for all $\varepsilon_{x} \in \left(-1,1\right)$ implies
$$\Delta\left(\varepsilon_{x},\varepsilon_{y}\right)<0$$ 
for all $\varepsilon_{x},\varepsilon_{y} \in \left(-1,1\right)$. Consequently, $p\left(\lambda\right)$ has one real root and two complex conjugate non-real roots: otherwise, $\Delta(\varepsilon_x,.)$ would have a real root between -1 and 1, in addition to real roots between $-\infty$ and -1 and between 1 and $+\infty$; hence it could not have two non-real roots.
 \end{proof}
 
\begin{proof} 
({\em of Theorem~\ref{thm:stab-C0}}) 
In Step 1\footnote{Step 1 has been made this simple by an anonymous reviewer whom we thank.}, we establish that the conditions for Lemma \ref{lem:cond(i)-(iii)} hold if and only if $\varepsilon_{x}+\varepsilon_{y}<0$. 
This is enough to prove part (i) according to Theorem \ref{thm:stab-indices}(a). 
In Step 2, when $\varepsilon_{x}+\varepsilon_{y}<0$, we use the function $F^{\textrm{index}}$ to calculate the stability index for each heteroclinic connection in the $C_{0}$-cycle. 
We show that both indices are positive; hence it follows from Theorem~\ref{thm:e.a.s.} that $C_0$ is essentially asymptotically stable.

\paragraph{Step 1:}

Consider the transition matrices $M^{\left(0\right)}$ and $M^{\left(1\right)}$ around the whole $C_{0}$-cycle in \eqref{eq:trans-matrix-C0} of Appendix \ref{subsec:trans-mat}. Since they are similar\footnote{We say that two square matrices of the same order, $A$ and $B$, are similar if there exists an invertible matrix $P$ such that $B=P^{-1}AP$. In particular, similar matrices have the same characteristic polynomial. In this case, we have $M_0^{-1}M^{(1)}M_0=M^{(0)}$.} conditions \emph{(i)--(ii)} of Lemma \ref{lem:cond(i)-(iii)} will simultaneously hold, or not hold true, for either matrices. 
Set then $M\equiv M^{(0)}$.

Lemma~\ref{lem:aux} establishes the type of eigenvalues of $M$.
Denote by $\lambda_1>0$ the real eigenvalue of $M$ and by $\alpha\pm\beta i$ the complex, non-real, eigenvalues of $M$, where $\alpha < 0$ is a consequence of the Routh-Hurwitz Criterion. 

Let $\lambda_{\max}$ be the maximum in absolute value root of $p(\lambda)$. Since
\begin{equation}\label{eq:det}
\det\left(M\right) = \lambda_1(\alpha^2+\beta^2)=1 \Leftrightarrow \lambda_1=\dfrac{1}{\alpha^2+\beta^2}
\end{equation}
it follows that 
$$
\lambda_1=\lambda_{\max} \Leftrightarrow \lambda_1>1.
$$
In order to prove that $\varepsilon_x+\varepsilon_y>0$ is equivalent to $\lambda_1 \neq \lambda_{\max}$, we note
 that, from \eqref{eq:tr} and \eqref{eq:B},
 $$
\textrm{Tr}\left(M\right)=\lambda_1 +2\alpha = \frac{-3-3\varepsilon_{x}-3\varepsilon_{y}+\varepsilon_{x}\varepsilon_{y}}{4}
$$
and 
$$
\textrm{B}\left(M\right)=2\alpha \lambda_1+\alpha^2+\beta^2 = \frac{-3+3\varepsilon_{x}+3\varepsilon_{y}+\varepsilon_{x}\varepsilon_{y}}{4},
$$
so that 
$$
\textrm{B}\left(M\right)-\textrm{Tr}\left(M\right)= 2\alpha (\lambda_1-1)+\alpha^2+\beta^2 - \lambda_1 = \frac{3}{2}(\varepsilon_x+\varepsilon_y).
$$
Using \eqref{eq:det}, we can write 
 $$
\alpha^2+\beta^2 - \lambda_1 = -\dfrac{1+\lambda_1}{\lambda_1}(\lambda_1-1)
$$
and thus,
$$
\left[ 2\alpha - \dfrac{1+\lambda_1}{\lambda_1} \right](\lambda_1-1) = \frac{3}{2}(\varepsilon_x+\varepsilon_y).
$$
Since $\alpha<0$ and $\lambda_1 >0$, the quantities $\lambda_1-1$ and $\varepsilon_x+\varepsilon_y$ have opposite signs. Hence, $\lambda_1 = \lambda_{\max}$ if and only if $\varepsilon_x+\varepsilon_y<0$.

\medskip

\paragraph{Step 2:}

To determine the stability index $\sigma_0$ along the heteroclinic connection $\left[\xi_1 \rightarrow \xi_0 \right]$ we examine the $\delta$-local basin of attraction of the $C_0$-cycle in a neighbourhood of  $\left[\xi_1 \rightarrow \xi_0 \right]$. The latter is defined as $\mathcal{B}_{\delta}^{\pi_{0}}$ in~\eqref{eq:basin-C0} of Appendix~\ref{sec:basin}.
We deduce that in the new coordinates~\eqref{eq:log-coord} the same is described by means of the matrices $M_0$ and $M^{(0)}=M_1M_{0}$. 
Exactly one entry of each $M_q$, $q=0,1$, is negative and so Theorem~\ref{thm:stab-indices} states that
\begin{align}
\sigma_{0} = & \min \Bigg\{ 
F^{\textrm{index}}\left(\boldsymbol{v}^{\max,0}\right),\label{eq:vmax} \\[.2cm]
& \begin{aligned}[c] \, \min \bigg\{& F^{\textrm{index}}\left(\dfrac{1-\varepsilon_{y}}{2},1,0\right), F^{\textrm{index}}\left(-\dfrac{1+\varepsilon_{x}}{2},0,1\right),\\[.1cm]
& F^{\textrm{index}}\left(1,0,0\right) \bigg\}, \end{aligned} \label{eq:rowM0} \\[.2cm]
& \begin{aligned}[c] \, \min \bigg\{& F^{\textrm{index}}\left(\dfrac{-1-3\varepsilon_x-\varepsilon_y+\varepsilon_x\varepsilon_y}{4},\dfrac{1-\varepsilon_x}{2},1\right),\\[.1cm]
& F^{\textrm{index}}\left(\dfrac{3+\varepsilon_y^2}{4},-\dfrac{1+\varepsilon_y}{2},0\right),
 F^{\textrm{index}}\left(\dfrac{1-\varepsilon_y}{2}, 1, 0\right)
\bigg\} \Bigg\}.
\end{aligned} \label{eq:rowM^0} 
\end{align}
The vector
$\boldsymbol{v}^{\max,0}$ in~\eqref{eq:vmax} is the row of the change of basis matrix from the basis of eigenvectors for $M^{(0)}$ to the canonical basis for $\R^3$ in the position associated with $\lambda_{\max}$.
Moreover, \eqref{eq:rowM0} takes the rows of $M_0$ while \eqref{eq:rowM^0} takes the rows of $M^{(0)}$.

Simple algebra attests that $\boldsymbol{v}^{\max,0}$ is a constant multiple of the vector
$$
\left(\dfrac{\lambda_1}{\left(\alpha-\lambda_1\right)^2+\beta^2},
\dfrac{\frac{4}{3+\varepsilon_y^2}\left[\left(\alpha+\frac{1+\varepsilon_y}{2}\right)^2+\beta^2\right]+\frac{1-\varepsilon_x}{2}}{\left(\alpha-\lambda_1\right)^2+\beta^2},
\dfrac{\frac{4}{3+\varepsilon_y^2}}{\left(\alpha-\lambda_1\right)^2+\beta^2}
\right).
$$
Hence its entries are all non-negative for any $\varepsilon_x,\varepsilon_y\in\left(-1,1\right)$. According to the values of $F^{\textnormal{index}}$ in Appendix \ref{sec:Findex} it follows that
$$
F^{\textrm{index}}\left(\boldsymbol{v}^{\max,0}\right)=+\infty.
$$
Again, non-negative entries lead to
$$
F^{\textrm{index}}\left(\frac{1-\varepsilon_y}{2},1,0\right)=F^{\textrm{index}}\left(1,0,0\right)=+\infty.
$$
On the other hand, when at least one entry is negative, we get
\[
\begin{aligned}
F^{\textrm{index}}\left(-\frac{1+\varepsilon_{x}}{2},0,1\right) & = \frac{1-\varepsilon_x}{1+\varepsilon_x} \\[0.2cm]
F^{\textrm{index}}\left(\frac{3+\varepsilon_y^2}{4},-\frac{1+\varepsilon_y}{2},0\right)  & = \frac{\left(1-\varepsilon_y\right)^2}{2\left(1+\varepsilon_y\right)}
\end{aligned}
\]
and 
\[
\begin{aligned}
& F^{\textrm{index}}\left(\frac{-1-3\varepsilon_x-\varepsilon_y+\varepsilon_x\varepsilon_y}{4}, \frac{1-\varepsilon_x}{2},1\right)\\ 
& \qquad \qquad =\begin{cases}
+\infty, & \text{if }-1-3\varepsilon_x-\varepsilon_y+\varepsilon_x\varepsilon_y>0\\[0.2cm]
\dfrac{\left(5-\varepsilon_y\right)\left(1-\varepsilon_x\right)}{1+3\varepsilon_x+\varepsilon_y-\varepsilon_x\varepsilon_y}, & \text{if }-1-3\varepsilon_x-\varepsilon_y+\varepsilon_x\varepsilon_y<0.
\end{cases}
\end{aligned}
\]
A straightforward comparison shows that for all $\varepsilon_x,\varepsilon_y\in\left(-1,1\right)$
\[
F^{\textrm{index}}\left(\frac{-1-3\varepsilon_x-\varepsilon_y+\varepsilon_x\varepsilon_y}{4}, \frac{1-\varepsilon_x}{2},1\right)>F^{\textrm{index}}\left(-\frac{1+\varepsilon_{x}}{2},0,1\right)
\]
and
\[
\sigma_0=\min\left\{\dfrac{1-\varepsilon_x}{1+\varepsilon_x},\dfrac{\left(1-\varepsilon_y\right)^2}{2\left(1+\varepsilon_y\right)} 
\right\}>0.
\] 

The proof for $\sigma_1$ runs as before by interchanging $\varepsilon_x$ and $\varepsilon_y$ in the calculations so that 
for all $\varepsilon_x,\varepsilon_y\in\left(-1,1\right)$
\[
\sigma_1=\min\left\{\dfrac{1-\varepsilon_y}{1+\varepsilon_y},\dfrac{\left(1-\varepsilon_x\right)^2}{2\left(1+\varepsilon_x\right)} 
\right\}>0.
\]
\end{proof}

The stability of the remaining heteroclinic cycles in the heteroclinic network of the RSP game is given in Theorems \ref{thm:stab-C1}--\ref{thm:stab-C4}. The proofs are omitted as they are analogous to that of Theorem \ref{thm:stab-C0} using the appropriate transition matrices in Appendix \ref{subsec:trans-mat}. In the statement of the following results it is useful to define
\begin{eqnarray*}
b_1 & = & \left(5-\varepsilon_{x}\right)\varepsilon_{y}^{2}+\left(\varepsilon_{x}^{2}+10\varepsilon_{x}+1\right)\varepsilon_{y}-\left(1-\varepsilon_{x}\right)\left(4+5\varepsilon_{x}\right) \\
b_2 & = & \left(5+\varepsilon_{x}\right)\varepsilon_{y}^{2}+\left(-\varepsilon_{x}^{2}+10\varepsilon_{x}-1\right)\varepsilon_{y}-\left(1+\varepsilon_{x}\right)\left(4-5\varepsilon_{x}\right).
\end{eqnarray*}

\begin{theorem}\label{thm:stab-C1}
For the $C_{1}$-cycle of the RSP game and any $\varepsilon_x,\varepsilon_y\in\left(-1,1\right)$,
\begin{enumerate}
\item[(a)] if either
$\varepsilon_{x}+\varepsilon_{y}<0$, or
$b_1<0$, or $\varepsilon_{x}-\varepsilon_{y}>0$,
then all stability indices are~$-\infty$, and  the cycle is completely unstable;
\item[(b)] if 
$\varepsilon_{x}+\varepsilon_{y}>0$, and $b_1>0$, and 
$\varepsilon_{x}-\varepsilon_{y}<0$,
then the stability indices are
\[
\begin{aligned}
\sigma_{1} = & \dfrac{-4+\varepsilon_x+\left(3-\varepsilon_x\right)\varepsilon_y+\varepsilon_y^2}{\left(1-\varepsilon_x\right)\left(1+\varepsilon_y\right)}<0 \\
\sigma_{2} = & \min\left\{ \dfrac{\varepsilon_{y}-\varepsilon_{x}}{1-\varepsilon_{y}}, \dfrac{1+2\varepsilon_{x}+\varepsilon_{y}^2}{2\left(1-\varepsilon_{x}\right)}\right\}>0,
\end{aligned}
\]
and the cycle is fragmentarily asymptotically stable.
\end{enumerate}
\end{theorem}

Along the $C_{1}$-cycle player $Y$ never wins. When both $\varepsilon_{x}+\varepsilon_{y}>0$ and $\varepsilon_{x}-\varepsilon_{y}<0$, we have that $\varepsilon_y>0$ so that player $Y$'s payoff for a tie is greater than the average payoff.
It seems reasonable that 
player $Y$ may then settle for a better response, leading only to a tie, rather than switching to the best response. 
The set bounded by $b_1=0$ ensures $\varepsilon_y$ is high enough to sustain continually this choice. The analogous occurs for player $X$ along the $C_2$-cycle as shown in the following

\begin{theorem}\label{thm:stab-C2}
For the $C_{2}$-cycle of the RSP game and any $\varepsilon_x,\varepsilon_y\in\left(-1,1\right)$,
\begin{enumerate}
\item[(a)] if either 
$\varepsilon_{x}+\varepsilon_{y}<0$, or
$b_2<0$, or 
$\varepsilon_{x}-\varepsilon_{y}<0$,
then all stability indices are~$-\infty$, and  the cycle is completely unstable;
\item[(b)] if 
$\varepsilon_{x}+\varepsilon_{y}>0$, and
$b_2>0$, and 
$\varepsilon_{x}-\varepsilon_{y}>0$,
then the stability indices are
\[
\begin{aligned}
\sigma_{0} = & \dfrac{-4+\varepsilon_y+\left(3-\varepsilon_y\right)\varepsilon_x+\varepsilon_x^2}{\left(1-\varepsilon_y\right)\left(1+\varepsilon_x\right)}<0 \\
\sigma_{2} = & \min\left\{ \dfrac{\varepsilon_{x}-\varepsilon_{y}}{1-\varepsilon_{x}}, \dfrac{1+2\varepsilon_{y}+\varepsilon_{x}^2}{2\left(1-\varepsilon_{y}\right)}\right\}>0,
\end{aligned}
\]
and the cycle is fragmentarily asymptotically stable.
\end{enumerate}
\end{theorem}

We see that the $C_{1}$- and $C_{2}$-cycles are never essentially asymptotically stable making them difficult to detect in simulations. They are fragmentarily asymptotically stable in a subset of the complement of the stability region for the $C_{0}$-cycle in the two-parameter space, see Figure~\ref{fig:attraction}. 

\begin{theorem}\label{thm:stab-C3}
For the $C_{3}$-cycle of the RSP game, all stability indices are~$-\infty$ for any $\varepsilon_{x},\varepsilon_{y} \in \left(-1,1\right)$, and the cycle is completely unstable.
\end{theorem}

\begin{theorem}\label{thm:stab-C4}
For the $C_{4}$-cycle of the RSP game, all stability indices are~$-\infty$ for any $\varepsilon_{x},\varepsilon_{y} \in \left(-1,1\right)$, and the cycle is completely unstable.
\end{theorem}

Recall that along the $C_3$- and $C_4$-cycles one of the players switches actions in two steps, first to a better, not best, response and only then to a best response. 
Intuitively, these heteroclinic cycles never exhibit any kind of stability because the gain incurred as a result of deviating outweighs the possible loss. Recall that along $C_3$ player $X$ waits for two consecutive choices of player $Y$ before switching action. However, if player $X$ were to play immediately after the first choice of player $Y$, the $C_1$-cycle would be followed and the average payoff of player $X$ would increase. The opposite is observed along $C_4$: a deviation in the timing of play by player $Y$ results in an increase of $Y$'s average playoff by following $C_2$.
This justifies the absence of the $C_3$- and $C_4$-cycles from the numerical observations made by Sato~\emph{et~al.}~(2005). 

The admissible regions where the RSP cycles can be stable are depicted in Figure~\ref{fig:attraction}. 

\begin{figure}[!htb]
\centering{}\includegraphics{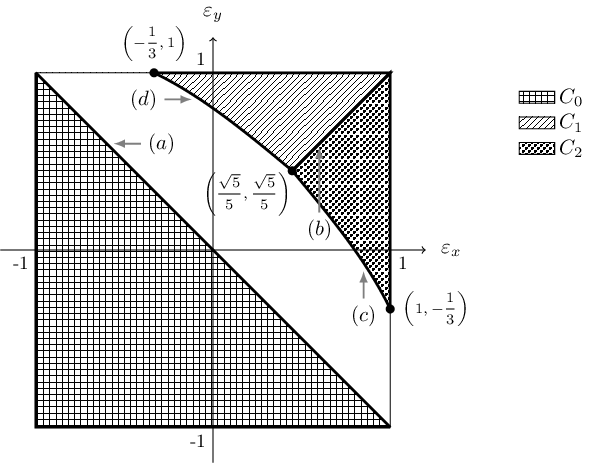}
\caption{\label{fig:attraction}Stability regions for the $C_0$-, $C_1$- and $C_2$-cycles in the two-parameter space. The lines in the figure are: (\emph{a}) $\varepsilon_{x}+\varepsilon_{y}=0$;
(\emph{b}) $\varepsilon_{x}-\varepsilon_{y}=0$; (\emph{c}) $\varepsilon_{y}=\frac{1-10\varepsilon_{x}+\varepsilon_{x}^{2}+\sqrt{81-24\varepsilon_{x}-2\varepsilon_{x}^{2}-40\varepsilon_{x}^{3}+\varepsilon_{x}^{4}}}{2\left(5+\varepsilon_{x}\right)}$;
(\emph{d}) $\varepsilon_{y}=\frac{-\left(1+10\varepsilon_{x}+\varepsilon_{x}^{2}\right)+\sqrt{81+24\varepsilon_{x}-2\varepsilon_{x}^{2}+40\varepsilon_{x}^{3}+\varepsilon_{x}^{4}}}{2\left(5-\varepsilon_{x}\right)}$.}
\end{figure}

Concerning other alternatives of play other than those studied above, we have conjectured them to be unstable. The reason for this is similar to that used to interpret the instability of $C_3$ and $C_4$. In fact, assuming that no player chooses an action that results in a loss, any other sequence of outcomes has to be achieved by a timing of play where players do not alternate in making their choice of action. As for $C_3$ and $C_4$ a deviation from this timing by the player that is supposed to wait leads to a more favourable outcome for this player.

\section{Concluding remarks}

We complete the study of the asymptotic behaviour in a RSP game governed by coupled replicator equations. It is a two-person (bimatrix) game with asymmetric players. 
Making use of recent developments in the study of dynamical systems, particularly in the study of stability of heteroclinic cycles, we classify three types of heteroclinic orbits for the two-person RSP game according to their stability as a function of two parameters. Such parameters describe the payoffs players receive when the outcome of their choice of actions is a tie. We allow them to range from almost as bad as a loss to almost as good as a win. 
Then we prove that if at least one player has a negative payoff for a tie, low enough that the sum of payoffs for a tie is itself negative, both players unilaterally avoid any choice of action leading to a tie. 
In the next stage game each player seeks the most favourable outcome towards the current opponent's strategy choice (the $C_0$-cycle). 
This is a situation that is stable in a strong sense for half of the two-parameter space. 
It is also consistent with the one-person (single-population) replicator dynamics where the heteroclinic cyclic on the boundary is asymptotically stable when it leads to payoffs that are higher than the equilibrium payoffs, and unstable if it leads to payoffs that are lower than the equilibrium payoffs.\footnote{We thank an anonymous reviewer for pointing this similarity out to us.}

On the other hand, the behavioural adjustment when the sum of payoffs for a tie is positive is not as stable. Even so it does exhibit some low level of stability if the payoffs for a tie are sufficiently high. There is now an incentive to play for a tie. 
This can be explained by the fact that if one player oscillates between a tie and a win, then its opponent that never wins is not satisfied with the outcome and tries to draw the game successively (the $C_1$- and $C_2$-cycles). 

We observe that the stability regions for the $C_0$-, $C_1$- and $C_2$-cycles are disjoint. The heteroclinic cycles, $C_3$ and $C_4$, for which play goes through all possible combinations of outcomes are never stable.

These stability results are consistent with numerical simulations and experiments referred throughout the current work and may help to clarify empirical examples of the RSP cycles discovered in nature and economics.

We note that this two-person RSP game is not zero-sum\footnote{We are, in fact, interested in the non-zero-sum game corresponding to $\varepsilon_x+\varepsilon_y\neq0$.}, consistent with Sigmund's (2011) concern that ``For most types of social and economic interactions, the assumption that the interests of the two players are always diametrically opposite does not hold.'' In addition, the payoff matrices reflect asymmetry, which is a feature that arises either in interpopulation or intrapopulation interactions. Social and economic dilemmas between consumers and sellers, firms and workers are typical examples where agents frequently adopt asymmetric positions. Also, differences in access to, and availability of, resources asymmetrically affect individuals' behaviour within each class of agents. 
We therefore contribute to a systematic treatment of the asymmetric version of replicator dynamics in the class of RSP games.  

Our results suggest that the two-person RSP game may be a good tool for modelling cyclic dominance where two fixed players independently exchange a winning position. For instance, the dynamics of the gasoline retail market by Noel~(2007) have reported that a major and an independent firms alternate in setting the highest price (see Figure~1 therein). If we assume that consumers buy at the lowest price, the firms can be regarded as switching to pure best replies in a two-person RSP game where the actions are ``fix a low price''~(R), ``fix an intermediate price''~(S) and ``fix a high price''~(P).  
This corresponds to the win-loss pattern of the $C_0$-cycle.
Hopkins and Seymour~(2002) have in turn shown that the existence of informed consumers precludes price dispersion. This is the case, however, in a model where only the firms are players. The two-person game presented here can complement the information provided in Hopkins and Seymour~(2002) by including consumers as active players in an equal footing to the firms. Price dispersion then appears not because different firms set different prices but because firms choose different prices over time. This is in line with the temporal price dispersion of Varian~(1980).

We end with a remark on the relation of our results to best-response dynamics: Hofbauer {\em et al.} (2009) have shown that the time-average of replicator dynamics are perturbed solutions of best-response dynamics. The time-average of trajectories converging to a heteroclinic cycle in replicator dynamics corresponds to trajectories of best-response dynamics converging to a periodic orbit (a Shapley orbit). Although we do not know the dynamics when all cycles are unstable, the results of van Strien and Sparrow (2011) indicate that very complex replicator dynamics could occur.\footnote{The second author and S.\ van Strien have work in preparation concerning this point.}

\paragraph{Acknowledgements:}
The second author is grateful to A.\ Rucklidge for an interesting conversation.
The  two authors were partially supported by CMUP (UID/MAT/00144/2013), which is funded by FCT (Portugal) with national (MEC) and European structural funds (FEDER), under the partnership agreement PT2020. L.\ Garrido-da-Silva is the recipient of the doctoral grant PD/BD/105731/2014 from FCT (Portugal).

\paragraph{Conflict of interest:} The authors declare that they have no conflict of interest.

\appendix

\section{Transitions near the RSP cycles}\label{sec:transitions}

In this section we describe the construction of Poincar\'e maps (also called return maps) from and to cross sections of the flow near each node once around an entire heteroclinic cycle. The Poincar\'e maps are the composition of local and global maps.
The local maps approximate the flow in a neighbourhood of a node. The global maps approximate the flow along a heteroclinic connection between two consecutive nodes. 

Near $\xi_j$ we introduce an incoming section $H_{j}^{in,i}$ across the heteroclinic connection $\left[\xi_{i}\rightarrow\xi_{j}\right]$ and an outgoing section $H_{j}^{out,k}$ across the heteroclinic connection $\left[\xi_{j}\rightarrow\xi_{k}\right]$, $j \neq i,k$. By definition, these are five-dimensional subspaces in $\R^6$. Krupa and Melbourne~(2004) have shown that not all dimensions are important in the study of stability of heteroclinic cycles as followed.

\subsection{Poincar\'e maps}

Assume that the flow is linearisable about each node (see Proposition 4.1. in Aguiar and Castro (2010) for detailed conditions). Locally at $\xi_j$, we denote by $-c_{ji}<0$ the eigenvalue in the stable direction through the heteroclinic connection $\left[\xi_{i}\rightarrow\xi_{j}\right]$ and $e_{jk}>0$ the eigenvalue in the unstable direction through the heteroclinic connection $\left[\xi_{j}\rightarrow\xi_{k}\right]$. As illustrated in Figure~\ref{fig:eigen-sec}(a) each node has two incoming connections and two outgoing connections.
Looking at $\xi_j$ from the point of view of the  sequence of heteroclinic connections $\left[\xi_i\rightarrow\xi_j\rightarrow\xi_k\right]$, we say that $-c_{ji}$ is \emph{contracting}, $e_{jk}$ is \emph{expanding}, $-c_{jl}$ and $e_{jm}$ are \emph{transverse} with $j \neq i,k,l,m$, $i \neq l$ and $k \neq m$ (see Garrido-da-Silva and Castro~2019; Krupa and Melbourne~2004;  Podvigina~2012; Podvigina and Ashwin~2011).

\begin{figure}[!htb]
\centering{}
\subfloat[]{\includegraphics{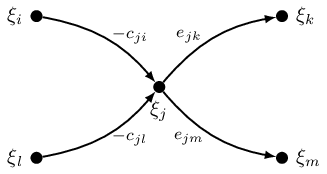}}
\hfill{}
\subfloat[]{\includegraphics{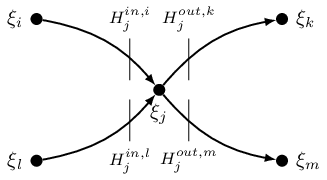}}
\caption{\label{fig:eigen-sec}Representation of (a) the eigenvalues and (b) the cross sections near a node $\xi_j$ with $j \neq i,k,l,m$, $i \neq l$ and $k \neq m$.}
\end{figure}

The linearised flow in the relevant local coordinates near $\xi_j$ is given by
\begin{equation}
\begin{aligned}
\dot{v}= & \; -c_{ji}v\\
\dot{w}= &\; e_{jk}w\\
\dot{z}_{1}= &\; -c_{jl}z_{1}\\
\dot{z}_{2}= &\; e_{jm}z_{2},
\end{aligned}
\label{eq:lin-flow}
\end{equation}
such that $v$, $w$ and $(z_1,z_2)$ correspond, respectively, to the contracting, expanding and transverse directions.
Table \ref{tab:eigen} provides all eigenvalues restricted to these directions for the three nodes $\xi_0$, $\xi_1$  and $\xi_2$.

All cross sections are reduced to a three-dimensional subspace and can be expressed as (see Figure~\ref{fig:eigen-sec}(b))
\[
\begin{aligned}
H_{j}^{in,i} & =\left\{ \left(1,w,z_{1},z_{2}\right):\; 0 \leq w,z_{1},z_{2}<1\right\} \\
H_{j}^{out,k} & =\left\{ \left(v,1,z_{1},z_{2}\right):\; 0\leq v,z_{1},z_{2}<1\right\} .
\end{aligned}
\]

We construct \emph{local maps} $\phi_{ijk}:H_{j}^{in,i}\rightarrow H_{j}^{out,k}$ near each $\xi_j$, \emph{global maps} $\psi_{jk}:H_{j}^{out,k}\rightarrow H_{k}^{in,j}$ near each heteroclinic connection $\left[\xi_{j}\rightarrow\xi_{k}\right]$, and their compositions $g_{j}=\psi_{jk} \circ \phi_{ijk}:H_{j}^{in,i}\rightarrow H_{k}^{in,j}$, $j \neq i,k$. Composing the latter successively along an entire heteroclinic cycle  yields the \emph{Poincar\'e maps} $\pi_j:H_j^{in,i} \rightarrow H_{j}^{in,i}$, one for each heteroclinic connection belonging to the heteroclinic cycle.  

Integrating \eqref{eq:lin-flow} we find
$$
\phi_{ijk}\left(w,z_{1},z_{2}\right)=\left(w^{{\textstyle \frac{c_{ji}}{e_{jk}}}},\; z_{1}w^{{\textstyle \frac{c_{jl}}{e_{jk}}}},\; z_{2}w^{-{\textstyle \frac{e_{jm}}{e_{jk}}}}\right), \quad \textrm{for }0<z_2<w^{\textstyle \frac{e_{jm}}{e_{jk}}}.
$$
 
\begin{table}
\begin{centering}
\begin{tabular}{l | l l l l}
$\xi_{0}\quad$ & $e_{01}=1$ 						& $e_{02}=\dfrac{1+\varepsilon_{x}}{2}$ 	& $-c_{01}=-1$ 							& $-c_{02}=-\dfrac{1-\varepsilon_{y}}{2}$ \\[10pt]
\hline  
$\xi_{1}\quad$ & $e_{12}=\dfrac{1+\varepsilon_{y} \T}{2}$ 	& $e_{10}=1$ 					& $-c_{12}=-\dfrac{1-\varepsilon_{x}}{2}$	& $-c_{10}=1$ \\[10pt]
\hline
$\xi_{2}\quad$ & $e_{20}=\dfrac{1-\varepsilon_{y} \T}{2}$ 		& $e_{21}=\dfrac{1-\varepsilon_{x}}{2}$ 	& $-c_{20}=-\dfrac{1+\varepsilon_{x} }{2}$ 	& $-c_{21}=-\dfrac{1-\varepsilon_{x}}{2}$ \\
\end{tabular}
\par\end{centering}
\caption{\label{tab:eigen}Eigenvalues of the linearisation of the flow about each node in a system of local coordinates in the basis of the associated contracting, expanding and transverse eigenvectors.}
\end{table}

On the other hand, expressions for global maps depend both on which heteroclinic connection and heteroclinic cycle one considers. Following the Remark on p.\ 1603 of Aguiar and Castro (2010), in the leading order any global map $\psi_{jk}$ is well represented by a permutation.

We describe the details for the cycle $C_0=\left[\xi_0\rightarrow\xi_1\rightarrow\xi_0\right]$.
The other cases are similar, and therefore, we omit the calculations.
Notice that 
$$
\Gamma\left((\textrm{R},\textrm{P}) \rightarrow (\textrm{S},\textrm{P}) \rightarrow (\textrm{S},\textrm{R}) \right) =C_0.
$$
We pick, respectively, the heteroclinic connections $\left[(\textrm{R},\textrm{P}) \rightarrow (\textrm{S},\textrm{P}) \right]$ and\linebreak$\left[(\textrm{S},\textrm{P}) \rightarrow (\textrm{S},\textrm{R}) \right]$ as representatives of $\left[\xi_0 \rightarrow \xi_1\right]$ and $\left[\xi_1 \rightarrow \xi_0\right]$, see Table \ref{tab:representative}. Considering the flow linearised about each representative heteroclinic connection the global maps have the form
\[
\begin{aligned}
\psi_{01}:& H_{0}^{out,1} \rightarrow H_{1}^{in,0}, \quad & \psi_{10} \left(v,z_1,z_2\right) & =\left(z_1,z_2,v\right) \\
\psi_{10}:& H_{1}^{out,0} \rightarrow  H_{0}^{in,1}, \quad & \psi_{01} \left(v,z_1,z_2\right)  & =\left(z_1,z_2,v\right).
\end{aligned}
\]
The pairwise composite maps $g_{0}=\psi_{01} \circ \phi_{101}$ and $g_{1}=\psi_{10} \circ \phi_{010}$ are
\begin{equation}
\begin{aligned}
g_{0}: & H_{0}^{in,1} \rightarrow H_{1}^{in,0}, \quad & g_{0}\left(w,z_1,z_2\right) & = \left(z_{1}w^{{\textstyle \frac{1-\varepsilon_y}{2}}},\; z_{2}w^{-{\textstyle \frac{1+\varepsilon_x}{2}}}, \; w \right), \\
& & \textrm{for } &  0<z_{2}<w^{\textstyle \frac{1+\varepsilon_x}{2}}, \\
g_{1}: & H_{1}^{in,0} \rightarrow H_{0}^{in,1}, \quad & g_{1}\left(w,z_1,z_2\right) & = \left( z_{1}w^{{\textstyle \frac{1-\varepsilon_x}{2}}},\; z_{2}w^{-{\textstyle \frac{1+\varepsilon_y}{2}}},\; w \right), \\
& & \textrm{for } &  0<z_{2}<w^{\textstyle \frac{1+\varepsilon_y}{2}}.
\end{aligned}
\label{eq:C0-gj}
\end{equation}
The dynamics in the vicinity of the $C_0$-cycle is accurately approximated by the two Poincar\'e maps $\pi_{0}=g_{1} \circ g_{0}$ and $\pi_{1}=g_{0} \circ g_{1}$ with
\begin{equation}
\begin{aligned}
\pi_{0} & : H_{0}^{in,1} \rightarrow   H_{0}^{in,1}, \\ 
\pi_{0} & \left(w,z_1,z_2\right) = \left(z_2 z_{1}^{\textstyle \frac{1-\varepsilon_x}{2}}w^{\textstyle \frac{-1-3\varepsilon_x-\varepsilon_y+\varepsilon_x \varepsilon_y}{4}},\; z_{1}^{-{\textstyle \frac{1+\varepsilon_y}{2}}}w^{\textstyle \frac{3+\varepsilon_y^2}{4}}, \; z_1 w^{\textstyle \frac{1-\varepsilon_y}{2}} \right),
\end{aligned}
\label{eq:pi0}
\end{equation}
for $0<z_{2}<w^{\textstyle \frac{1+\varepsilon_x}{2}}$ and $z_1>w^{\textstyle \frac{3+\varepsilon_y^2}{2\left(1+\varepsilon_y\right)}}$,
\begin{equation}
\begin{aligned}
\pi_{1} & : H_{1}^{in,0} \rightarrow  H_{1}^{in,0}, \\
\pi_{1} & \left(w,z_1,z_2\right) =  \left(z_2 z_{1}^{\textstyle \frac{1-\varepsilon_y}{2}}w^{\textstyle \frac{-1-\varepsilon_x-3\varepsilon_y+\varepsilon_x \varepsilon_y}{4}},\; z_{1}^{-{\textstyle \frac{1+\varepsilon_x}{2}}}w^{\textstyle \frac{3+\varepsilon_y^2}{4}}, \; z_1 w^{\textstyle \frac{1-\varepsilon_x}{2}} \right),
\end{aligned}
\label{eq:pi1}
\end{equation}
for $0<z_{2}<w^{\textstyle \frac{1+\varepsilon_y}{2}}$ and $z_1>w^{\textstyle \frac{3+\varepsilon_x^2}{2\left(1+\varepsilon_x\right)}}$.

\subsection{Transition matrices}\label{subsec:trans-mat}

Consider the change of coordinates 
\begin{equation}
\boldsymbol{\eta}\equiv\left(\eta_{1},\eta_{2},\eta_{3}\right)=\left(\ln v,\ln z_{1},\ln z_{2} \right).
\label{eq:log-coord}
\end{equation}
The maps $g_{j}:H_j^{in,i}\rightarrow H_k^{in,j}$, $j \neq i,k$, become linear 
$$
g_j\left(\boldsymbol{\eta}\right)=M_j  \boldsymbol{\eta}
$$
and $M_j$ are called the \emph{basic transition matrices}. 
For the Poincar\'e maps $\pi_j:H_j^{in,i} \rightarrow H_j^{in,i}$ the transition matrices are the product of basic transition matrices in the appropriate order. We denote them by $M^{(j)}$. 

The basic transition matrices of the maps $g_0$ and $g_1$ in \eqref{eq:C0-gj} with respect to $C_0$ are  
\[
M_{0} =\left[\begin{array}{ccc}
\dfrac{1-\varepsilon_y}{2} & 1 & 0\\[.4cm]
-\dfrac{1+\varepsilon_x}{2} & 0 & 1\\[.4cm]
1 & 0 & 0
\end{array}\right], \; \; \;  \;
M_{1} =\left[\begin{array}{ccc}
\dfrac{1-\varepsilon_x}{2} & 1 & 0\\[.4cm]
-\dfrac{1+\varepsilon_y}{2} & 0 & 1\\[.4cm]
1 & 0 & 0
\end{array}\right].
\]
Then, the products $M^{(0)}=M_1 M_0$ and $M^{(1)}=M_0 M_1$ provide the transition matrices of the Poincar\'e maps $\pi_0$ in \eqref{eq:pi0} and $\pi_1$ in \eqref{eq:pi1}:
\begin{equation}
\begin{aligned}
M^{(0)} & =\left[\begin{array}{ccc}
\dfrac{-1-3\varepsilon_x-\varepsilon_y+\varepsilon_x\varepsilon_y}{4} & \dfrac{1-\varepsilon_x}{2} & 1\\[.4cm]
\dfrac{3+\varepsilon_y^2}{4} & -\dfrac{1+\varepsilon_y}{2} & 0\\[.4cm]
\dfrac{1-\varepsilon_y}{2} & 1 & 0
\end{array}\right]\\[.3cm]
M^{(1)} & =\left[\begin{array}{ccc}
\dfrac{-1-\varepsilon_x-3\varepsilon_y+\varepsilon_x\varepsilon_y}{4} & \dfrac{1-\varepsilon_y}{2} & 1\\[.4cm]
\dfrac{3+\varepsilon_x^2}{4} & -\dfrac{1+\varepsilon_x}{2} & 0\\[.4cm]
\dfrac{1-\varepsilon_x}{2} & 1 & 0
\end{array}\right].
\end{aligned}
\label{eq:trans-matrix-C0}
\end{equation}

\begin{figure}[!htb]
\centering
\subfloat[]{\includegraphics{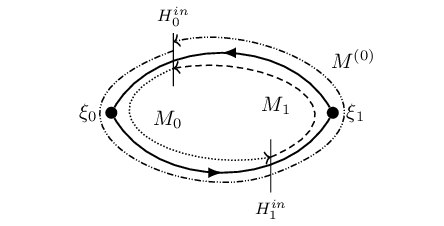}}
\subfloat[]{\includegraphics{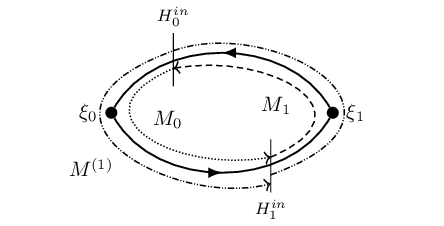}}
\caption{\label{fig:trans-map}Definition of the transition matrices (a) $M^{(0)}=M_1M_0$ and (b) $M^{(1)}=M_0M_1$ around the $C_0$-cycle. The matrix $M_0$ corresponds to the transformation represented by a dotted line and $M_1$ to that represented by a dashed line.}
\end{figure}

Analogously, this process yields the transition matrices for the remaining heteroclinic cycles. We use different accents according to the heteroclinic cycle: 
for $C_1$,
\[
\begin{aligned}
\widetilde{M}^{(1)}: &  H_1^{in,2} \rightarrow  H_1^{in,2}, \quad & \widetilde{M}^{(1)} & = \widetilde{M}_2 \widetilde{M}_1\\
\widetilde{M}^{(2)}: &  H_2^{in,1} \rightarrow  H_2^{in,1}, \quad & \widetilde{M}^{(2)} & =\widetilde{M}_1 \widetilde{M}_2
\end{aligned}
\]
where
\[
\widetilde{M}_{1} =\left[\begin{array}{ccc}
\dfrac{2}{1+\varepsilon_y} & 1 & 0\\[.4cm]
\dfrac{1-\varepsilon_x}{1+\varepsilon_y} & 0 & 0\\[.4cm]
-\dfrac{2}{1+\varepsilon_y} & 0 & 1
\end{array}\right], \; \; \;  \;
\widetilde{M}_{2} =\left[\begin{array}{ccc}
-\dfrac{1-\varepsilon_y}{1-\varepsilon_x} & 0 & 1\\[.4cm]
\dfrac{1+\varepsilon_x}{1-\varepsilon_x} & 1 & 0\\[.4cm]
\dfrac{1+\varepsilon_y}{1-\varepsilon_x} & 0 & 0
\end{array}\right];
\]
for $C_2$, 
\[
\begin{aligned}
\widetilde{\widetilde{M}}^{(0)}: &  H_0^{in,2} \rightarrow  H_0^{in,2}, \quad & \widetilde{\widetilde{M}}^{(0)}& = \widetilde{\widetilde{M}_2} \widetilde{\widetilde{M}_0}\\
\widetilde{\widetilde{M}}^{(2)}: &  H_2^{in,0} \rightarrow  H_2^{in,0}, \quad & \widetilde{\widetilde{M}}^{(2)}& = \widetilde{\widetilde{M}_0} \widetilde{\widetilde{M}_2}
\end{aligned}
\]
where
\[
\widetilde{\widetilde{M}_{0}} =\left[\begin{array}{ccc}
\dfrac{2}{1+\varepsilon_x} & 1 & 0\\[.4cm]
\dfrac{1-\varepsilon_y}{1+\varepsilon_x} & 0 & 0\\[.4cm]
-\dfrac{2}{1+\varepsilon_x} & 0 & 1
\end{array}\right], \; \; \;  \;
\widetilde{\widetilde{M}_{2}} =\left[\begin{array}{ccc}
-\dfrac{1-\varepsilon_x}{1-\varepsilon_y} & 0 & 1\\[.4cm]
\dfrac{1+\varepsilon_y}{1-\varepsilon_y} & 1 & 0\\[.4cm]
\dfrac{1+\varepsilon_x}{1-\varepsilon_y} & 0 & 0
\end{array}\right];
\]
for $C_3$, 
\[
\begin{aligned}
\widehat{M}^{(1)}: &  H_1^{in,0} \rightarrow  H_1^{in,0}, \quad & \widehat{M}^{(1)} & = \widehat{M}_2 \widehat{M}_0 \widehat{M}_1\\
\widehat{M}^{(2)}: &  H_2^{in,1} \rightarrow  H_2^{in,1}, \quad & \widehat{M}^{(2)} & =\widehat{M}_0 \widehat{M}_1 \widehat{M}_2\\
\widehat{M}^{(0)}: &  H_0^{in,2} \rightarrow  H_0^{in,2}, \quad & \widehat{M}^{(0)} & =\widehat{M}_1 \widehat{M}_2 \widehat{M}_0
\end{aligned}
\]
where
\[
\widehat{M}_{1} =\left[\begin{array}{ccc}
-\dfrac{2}{1+\varepsilon_y} & 0 & 1\\[.4cm]
\dfrac{1-\varepsilon_x}{1+\varepsilon_y} & 1 & 0\\[.4cm]
\dfrac{2}{1+\varepsilon_y} & 0 & 0
\end{array}\right], \; \; \;  \;
\widehat{M}_{2} =\left[\begin{array}{ccc}
\dfrac{1+\varepsilon_x}{1-\varepsilon_y} & 1 & 0\\[.4cm]
\dfrac{1+\varepsilon_y}{1-\varepsilon_y} & 0 & 0\\[.4cm]
-\dfrac{1-\varepsilon_x}{1-\varepsilon_y} & 0 & 1
\end{array}\right],
 \; \; \;  \;
 \widehat{M}_{0} =\left[\begin{array}{ccc}
1 & 1 & 0\\[.4cm]
-\dfrac{1+\varepsilon_x}{2} & 0 & 1\\[.4cm]
\dfrac{1-\varepsilon_y}{2} & 0 & 0
\end{array}\right];
\]
and, for $C_4$, 
\[
\begin{aligned}
\widehat{\widehat{M}}^{(0)}: &  H_0^{in,1} \rightarrow  H_0^{in,1}, \quad & \widehat{\widehat{M}}^{(0)} & = \widehat{\widehat{M}_2} \widehat{\widehat{M}_1} \widehat{\widehat{M}_0}\\
\widehat{\widehat{M}}^{(2)}: &  H_2^{in,0} \rightarrow  H_2^{in,0}, \quad & \widehat{\widehat{M}}^{(2)} & = \widehat{\widehat{M}_1} \widehat{\widehat{M}_0} \widehat{\widehat{M}_2}\\
\widehat{\widehat{M}}^{(1)}: &  H_1^{in,2} \rightarrow  H_1^{in,2}, \quad & \widehat{\widehat{M}}^{(1)} & = \widehat{\widehat{M}_0} \widehat{\widehat{M}_2} \widehat{\widehat{M}_1}
\end{aligned}
\]
where
\[
\widehat{\widehat{M}_{0}} =\left[\begin{array}{ccc}
-\dfrac{2}{1+\varepsilon_x} & 0 & 1\\[.4cm]
\dfrac{1-\varepsilon_y}{1+\varepsilon_x} & 1 & 0\\[.4cm]
\dfrac{2}{1+\varepsilon_x} & 0 & 0
\end{array}\right], \; \; \;  \;
\widehat{\widehat{M}_{2}} =\left[\begin{array}{ccc}
\dfrac{1+\varepsilon_y}{1-\varepsilon_x} & 1 & 0\\[.4cm]
\dfrac{1+\varepsilon_x}{1-\varepsilon_x} & 0 & 0\\[.4cm]
-\dfrac{1-\varepsilon_y}{1-\varepsilon_x} & 0 & 1
\end{array}\right],
 \; \; \;  \;
\widehat{\widehat{M}_{1}} =\left[\begin{array}{ccc}
1 & 1 & 0\\[.4cm]
-\dfrac{1+\varepsilon_y}{2} & 0 & 1\\[.4cm]
\dfrac{1-\varepsilon_x}{2} & 0 & 0
\end{array}\right].
\]

\section{The stability index}\label{sec:basin}

Each Poincar\' e map $\pi_j:\R^3\rightarrow \R^3$ induces a discrete dynamical system through the relation $\boldsymbol{x}_{k+1}=\pi_j \left(\boldsymbol{x}_k\right)$, where $\boldsymbol{x}_k=\left(w_k,z_{1,k},z_{2,k}\right)$ denotes the state at the discrete time $k$. 
The stability of a heteroclinic cycle then follows from the stability of the fixed point at the origin of Poincar\'e maps around the former. 

As in Podvigina~(2012) and Podvigina and Ashwin~(2011), for $\delta>0$ let $\B_{\delta}^{\pi_j}$ be the $\delta$-local basin of attraction of $\boldsymbol{0}\in\R^3$ for the map $\pi_j$. Roughly speaking, $\B_{\delta}^{\pi_j}$ is the set of all initial conditions near $\xi_j$ whose trajectories remain in a $\delta$-neighbourhood of the heteroclinic cycle and converge to it. 

We use $\left\Vert \cdot \right\Vert$ to express Euclidean norm on $\R^3$.
Taking, for example, the Poincar\'e map $\pi_0$~in~\eqref{eq:pi0} with respect to the $C_0$-cycle the local basin $\B_{\delta}^{\pi_0}$ is defined to be
\begin{equation}
\begin{aligned}
\B_{\delta}^{\pi_0}=\Big\{ \boldsymbol{x} \in \R^3 :  & \left\Vert \pi_0^k(\boldsymbol{x}) \right\Vert < \delta,  \; \left\Vert g_0 \circ \pi_0^k(\boldsymbol{x}) \right\Vert < \delta \textrm { for all } k\in\N_0   \\
\textrm{ and } & \lim_{k \rightarrow \infty} \left\Vert \pi_0^k(\boldsymbol{x})\right\Vert = 0, \; \lim_{k \rightarrow \infty} \left\Vert g_0 \circ \pi_0^k(\boldsymbol{x})\right\Vert = 0 \Big\}.
\end{aligned}
\label{eq:basin-C0}
\end{equation}

In the new coordinates $\boldsymbol{\eta}$ \eqref{eq:log-coord} the origin in $\R^3$ becomes $-\boldsymbol{\infty}$. For asymptotically small $\boldsymbol{x}=\left(w,z_1,z_2\right)\in\B_{\delta}^{\pi_j}$ the requirement (as in~\eqref{eq:basin-C0}) that the iterates $\pi_j^k(\boldsymbol{x})$ approach $\boldsymbol{0}$ (hence the heteroclinic cycle) as $k\rightarrow\infty$
corresponds to all asymptotically large negative $\boldsymbol{\eta}$ such that
\begin{equation}
\lim_{k\rightarrow\infty} \left(M^{(j)}\right)^k\boldsymbol{\eta}=-\boldsymbol{\infty}.
\label{eq:U-inf}
\end{equation}

For $M=M^{(j)}$ we denote the set of points satisfying~\eqref{eq:U-inf} by $U^{-\infty}\left(M\right)$.
Lemma~3 in Podvigina~(2012), together with its reformulation as Lemma~3.2 in Garrido-da-Silva and Castro~(2019), provide necessary and sufficient conditions for $U^{-\infty}\left(M\right)$  having a positive measure. The conditions depend on the dominant eigenvalue and the associated eigenvector of $M$. We transcribe the result in its useful form in the following

\begin{lemma}[adapted from Lemma 3 in Podvigina (2012)]
Let $\lambda_{\max}$ be the maximum, in absolute value, eigenvalue of the matrix $M:\R^N \rightarrow \R^N$ and $\boldsymbol{w}^{\max}=\left(w_1^{\max},\ldots,w_N^{\max}\right)$ be an associated eigenvector. Suppose $\lambda_{\max} \neq 1$. The measure $\ell\left(U^{-\infty}(M)\right)$ is positive if and only if the three following conditions are satisfied:\label{lem:cond(i)-(iii)}
\begin{enumerate}
\item[(i)] $\lambda_{\max}$ is real;
\item[(ii)] $\lambda_{\max}>1$;
\item[(iii)] $w_{l}^{\max}w_{q}^{\max}>0$ for all $l,q=1,\ldots N$.
\end{enumerate}
\end{lemma}

\subsection{The function $F^{\textrm{index}}$}\label{sec:Findex}

The function $F^{\textrm{index}}:\R^N \rightarrow \R$ used to calculate the stability indices along heteroclinic connections is constructed in Garrido-da-Silva and Castro~(2019). 

For $N=3$ and any nonzero $\boldsymbol{\alpha}=\left(\alpha_1,\alpha_2,\alpha_3\right) \in \R^3$ denote $\alpha_{\min}=\min\left\{\alpha_1,\alpha_2,\alpha_3\right\}$ and $\alpha_{\max}=\max\left\{\alpha_1,\alpha_2,\alpha_3\right\}$. From Appendix~A.1 in Garrido-da-Silva and Castro~(2019) we have
\[
F^{\textrm{index}}(\boldsymbol{\alpha})=F^{+}(\boldsymbol{\alpha})-F^{-}(\boldsymbol{\alpha})
\]
with $F^{-}(\boldsymbol{\alpha})=F^{+}(-\boldsymbol{\alpha})$ where
\[
F^{+}(\boldsymbol{\alpha})=\begin{cases}
+\infty, & \textrm{if }\alpha_{\min} \geq0\\
0, & \textrm{if }\alpha_{1}+\alpha_{2}+\alpha_{3}\leq 0\\
-\dfrac{\alpha_{1}+\alpha_{2}+\alpha_{3}}{\alpha_{\min}}, & \textrm{if }\alpha_{\min} <0\textrm{ and }\alpha_{1}+\alpha_{2}+\alpha_{3}\geq 0,
\end{cases}
\]
and 
\[
F^{-}\left(\boldsymbol{\alpha}\right)=\begin{cases}
+\infty, & \textrm{if }\alpha_{\max} \leq0\\
0, & \textrm{if }\alpha_{1}+\alpha_{2}+\alpha_{3}\geq 0\\
-\dfrac{\alpha_{1}+\alpha_{2}+\alpha_{3}}{\alpha_{\max}}, & \textrm{if }\alpha_{\max} >0\textrm{ and }\alpha_{1}+\alpha_{2}+\alpha_{3}\leq 0.
\end{cases}
\]
It then follows that 
\[
F^{\textrm{index}}\left(\boldsymbol{\alpha}\right)=\begin{cases}
+\infty, & \textrm{if }\alpha_{\min} \geq0\\
-\infty, & \textrm{if }\alpha_{\max} \leq0\\
0, & \textrm{if }\alpha_{1}+\alpha_{2}+\alpha_{3}=0\\
\dfrac{\alpha_{1}+\alpha_{2}+\alpha_{3}}{\alpha_{\max}}, & \textrm{if }\alpha_{\max} >0\textrm{ and }\alpha_{1}+\alpha_{2}+\alpha_{3}<0\\[0.3cm]
-\dfrac{\alpha_{1}+\alpha_{2}+\alpha_{3}}{\alpha_{\min}}, & \textrm{if }\alpha_{\min} <0\textrm{ and }\alpha_{1}+\alpha_{2}+\alpha_{3}>0.
\end{cases}
\]

\end{document}